\documentclass[letterpaper,12pt,reqno]{amsart}
\usepackage{amssymb}
\usepackage{cases}
\usepackage{amsmath}
\usepackage{mathrsfs}
\usepackage{amsfonts}
\usepackage{amsfonts}
\usepackage[all]{xy}
\usepackage{color}
\DeclareFontFamily{OT1}{pzc}{}
\DeclareFontShape{OT1}{pzc}{m}{it}{<-> [1.15] rpzcmi}{}
\DeclareMathAlphabet{\mathzc}{OT1}{pzc}{m}{it}

\def\End{\operatorname{End}\kern-.5pt}

\def\KK{{\mathzc K\kern0pt}}
\def\vv{{\mathzc v\kern.5pt}}

\def\aq{/\kern-2pt/}

\def\diag{\operatorname{diag}}

\def\La{{\Lambda}}
\def\fT{{\mathfrak T}}
\def\fS{{\mathfrak S}}
\def\End{{\text{\rm End}}}

\def\row{{\text{\rm row}}}

\def\bsi{{\boldsymbol i}}
\def\bsj{{\boldsymbol j}}

\def\bsq{{\boldsymbol q}}
\def\bse{{\boldsymbol e}}
\def\up{{\boldsymbol{\upsilon}}}
\def\sH{{\mathcal H}}
\def\sS{{\mathcal S}}
\def\sZ{{\mathcal Z}}
\def\sD{{\mathcal D}}

\def\sI{{\mathcal I}}

\def\la{{\lambda}}
\def\al{{\alpha}}

\def\vep{{\varepsilon}}

\def\xib{{\xi}}

\def\sfK{{\mathsf K}}
\def\sfE{{\mathsf E}}
\def\sfF{{\mathsf F}}

\def\bsfA{{\boldsymbol{\mathfrak A}}}
\def\fA{{\mathfrak A}}
\def\bfU{{\mathbf U}}
\def\bsS{{\boldsymbol{\sS}}}

\newtheorem{theorem}{Theorem}[section]
\newtheorem{lemma}[theorem]{Lemma}
\newtheorem{proposition}[theorem]{Proposition}
\newtheorem{corollary}[theorem]{Corollary}
\theoremstyle{definition}
\newtheorem{definition}[theorem]{Definition}
\newtheorem{example}[theorem]{Example}

\newtheorem{remark}[theorem]{Remark}

\numberwithin{equation}{theorem}

\def\co{{\rm co}}
\def\ro{{\rm ro}}

\def\vep{{\varepsilon}}

\def\hat{\widehat}

\def\mbq{{\mathbb Q}}
\def\mbz{{\mathbb Z}}

\def\bfl{{\mathbf 0}}

\def\fS{{\mathfrak S}}
\def\wh{\widehat}
\def\vph{{\varphi}}
\def\glmn{{\mathfrak{gl}_{m|n}}}
\def\row{{\text{\rm row}}}
\def\sX{{\mathcal X}}
\def\fm{{\mathfrak m}}
\def\cC{{\mathsf C}}
\def\sfc{{\mathsf c}}

\def\rc{{\rm rc}}
\def\sY{{\mathcal Y}}
\def\gd{{\text{gd}}}

\def\sfT{{\mathsf T}}
\def\sfS{{\mathsf S}}

\def\sfQ{{\mathsf Q}}

\def\sfC{{\mathsf C}}
\def\sfp{{\mathsf p}}
\def\sfq{{\mathsf q}}
\def\bfT{{\mathbf T}}

\def\fh{{\mathfrak h}}
\def\bsfh{{\boldsymbol{\fh}}}
\def\sh{{\text{\rm sh}}}

\begin{document}

\baselineskip16pt
\def\hei{\relax}

 \title[Canonical bases for  the quantum supergroups $\bfU(\mathfrak{gl}_{m|n})$]{Canonical bases for  the quantum\\ supergroups $\bfU(\mathfrak{gl}_{m|n})$}
\author{Jie Du and Haixia Gu}
\address{J.D., School of Mathematics and Statistics,
University of New South Wales, Sydney NSW 2052, Australia}
\email{j.du@unsw.edu.au}
\address{H.G., Department of Mathematics, East China Normal University, Shanghai 200062, and Huzhou University, Huzhou, Zhejiang 313000, China}
\email{alla0824@126.com}
%\date{\today}
\thanks{The authors gratefully acknowledge support from ARC under grant DP120101436 and ZJNSF (No. LZ14A010001). The work was completed while the second author was visiting UNSW}


\begin{abstract} We give a combinatorial construction for the canonical bases of the $\pm$-parts of the quantum enveloping superalgebra $\bfU(\mathfrak{gl}_{m|n})$ and discuss their relationship with the Kazhdan-Lusztig bases for the quantum Schur superalgebras $\bsS(m|n,r)$ introduced in \cite{DR}. We will also extend this relationship to the induced bases for simple polynomial representations of $\bfU(\mathfrak{gl}_{m|n})$.
\end{abstract}

 \maketitle
\section{Introduction}
The theory of Kazhdan--Lusztig bases for Iwahori-Hecke algebras and its subsequent generalisation by Lusztig to canonical bases for quantum groups and their integrable modules was an important breakthrough in representation theory. Remarkably, this theory can also be approached through Kashiwara's crystal and global crystal bases and thus results in more applications. For example, it serves as an important motivation for the categorification of quantum enveloping algebras since its geometric construction provides a first model of categorification.

Naturally, generalising the canonical basis (or crystal) theory to the quantum supergroups attracts lots of attention and becomes rather challenging. For example, Benkart--Kang--Kashiwara \cite{BKK} developed a crystal basis theory for a certain class of representations of the quantum general linear Lie superalgebras; while Clark--Hill--Wang \cite{CHW} constructed crystal/canonical bases for quantum supergroups with no isotropic odd roots which includes $\mathfrak{sop}(1|2n)$ as the only finite type example. More recently, building on the work of Leclerc \cite{Lec} on quantum shuffles algebras, they \cite{CHW2} established the existence of the canonical basis (called the pseudo-canonical basis loc. cit.) of a quantum supergroup of special type, including the quantum supergroups $\bfU(\mathfrak{gl}_{m|n})$.

In this paper, we will provide a new construction of the canonical basis for the most fundamental quantum supergroup $\bfU(\mathfrak{gl}_{m|n})$, different from the one given in \cite{CHW2}. This approach was motivated by the following. First, canonical bases have been constructed in \cite{DR} for quantum Schur superalgebras, which are quotients of $\bfU(\mathfrak{gl}_{m|n})$. Second, the quantum supergroup $\bfU(\mathfrak{gl}_{m|n})$ can be realised as a ``limit algebra'' of quantum Schur superalgebras \cite{DG}, which generalises the construction of quantum $\mathfrak{gl}_{n}$ by Beilinson, Lusztig and MacPherson \cite{BLM}.
Thus, it is natural to expect the existence of the canonical basis for (the $\pm$-parts of) $\bfU(\mathfrak{gl}_{m|n})$ as a ``limit basis''
of the canonical bases for quantum Schur superalgebras.

%We now use this realisation to construct the canonical basis for the $\pm$-parts of $\bfU(\mathfrak{gl}_{m|n})$.

The main discovery in the paper is the identification of the realisation bases of the $\pm$-parts with PBW type bases. It was observed by Du--Parshall \cite{DP} in the quantum $\mathfrak{gl}_{n}$ case that the BLM realisation bases for the $\pm$-parts share the same multiplication formula of a basis element by generators as the Ringel--Hall algebra of a linear quiver. In the super case, the nonexistence of Ringel--Hall algebras made us to seek a similar relation directly. Thus, under the realisation isomorphism, we prove in Theorem \ref{PBW basis} that the realisation basis for the $+$-part coincides with the PBW type bases considered in \cite{Z}. Now the realisation basis has a triangular relation to a certain monomial basis as discovered in the proof of Theorem \cite[Th.~8.1]{DG} via a similar relation in the quantum Schur superalgebras \cite[Th.~7.1]{DG}. Thus, we obtain a triangular relation between a monomial basis and a PBW basis. This relation is the key to the existence of the canonical bases (Theorem \ref{canonical basis}) and makes it computable, following an algorithm used in \cite{CX}. We will also see in Theorem \ref{canonical basis of two parts} how this canonical basis, as a ``limit basis'', is connected to the canonical bases of quantum Schur superalgebras

%Since the realisation of $\bfU(\mathfrak{gl}_{m|n})$ was built on the structure of quantum Schur superalgebras, it is natural to expect the canonical bases for the $\pm$-parts should have a connection to the canonical basis of a quantum Schur superalgebra constructed in \cite{DR}. We will reveal such a relation in Theorem \ref{canonical basis of two parts}.

The canonical basis for the negative part in the nonsuper case induces nicely canonical bases for simple representations of $\bfU(\mathfrak{gl}_{n})$. However, in the super case, this nice property is no longer true in general. Clark, Hill and Wang conjectured in \cite[Conj.~8.9]{CHW2} that the property should hold for $\bfU(\mathfrak{gl}_{m|1})$ and their polynomial representations. We will prove this part of their conjecture in Corollary~\ref{CHWConj}. In general,
we will show that, for a simple polynomial representations of $\bfU(\mathfrak{gl}_{m|n})$, any basis induced from the canonical basis of a quantum Schur superalgebra $\bsS(m|n,r)$ coincides with the one induced by the canonical basis of the negative part of $\bfU(\mathfrak{gl}_{m|n})$.

It would be interesting to identify the canonical bases introduced here with the pseudo-canonical bases introduced in \cite[(7.3)]{CHW2} and to make a comparison with the canonical basis for the quantum coordinate superalgebra given in \cite{ZZ, ZHC}.

We organise the paper as follows. We will collect the basic theory of quantum Schur superalgebras in \S2, including a construction of the canonical basis. We provide in \S3 some multiplication formulas of high order in order to construct the Lusztig type form of the $\pm$-parts and prove that its defining basis is nothing but a PBW type basis in \S4. In \S5, we construct the canonical bases for the $\pm$-parts and describe a relation between this basis and that for quantum Schur superalgebras. As examples, we compute the canonical bases for the supergroups $\bfU(\mathfrak{gl}_{2|1})$ and $\bfU(\mathfrak{gl}_{2|2})$. Finally, in the last section, we discuss simple polynomial representations of $\bfU(\mathfrak{gl}_{m|n})$ and relate their bases induced by the canonical bases of $\bsS(m|n,r)$ and of the negative part of $\bfU(\mathfrak{gl}_{m|n})$. As an application, we prove the conjecture \cite[Conj.~8.9]{CHW2} for polynomial representations.

Throughout, let $m,n$ be nonnegative integers, not both zero. For any integers $i,t,s$ with $0\leq t\leq s$, define\vspace{-1ex}
\begin{equation}
\hat{i}=\left\{
\begin{aligned}
&0,& \text{  if } 1\leq i\leq m;\\
&1,&\text{ if }m+1\leq i\leq m+n,
\end{aligned}
\right.\;\;\text{ and }\;\; \left[\!\!\left[t\atop s\right]\!\!\right]=\frac{[\![s]\!]^!}{[\![t]\!]^![\![s-t]\!]^!}=\up^{s(t-s)}\left[t\atop s\right],\vspace{-1ex}
\end{equation}
where $[\![r]\!]^{!}:=[\![1]\!][\![2]\!]\cdots[\![r]\!]$ with $[\![i]\!]=1+\up^2+\cdots+\up^{2(i-1)}$ and $[i]=\left[i\atop 1\right]=\frac{\up^i-\up^{-i}}{\up-\up^{-1}}$.

Let $\up$ be an indeterminate and let $\up_a=\up^{(-1)^{\hat a}}$ for all $1\leq a\leq m+n$.

\section{Canonical bases for quantum Schur superalgebras}
 Let $\fS_r$ be the symmetric group on $r$ letters and let $S=\{(k,k+1)\mid 1\leq k<r\}$ be the set of basic transpositions. Form the Coxeter system $(\fS_r,S)$ and denote the length function with respect to $S$ by
$l:W\to\mathbb{N}$ and the Bruhat order on $\fS_r$ by $\leq$.

An $N$-tuple
$\lambda=(\lambda_1,\lambda_2,\cdots,\lambda_N)\in\mathbb N^N$ of non-negative integers is
called a composition of $r$ into $N$ parts if
$|\la|:=\sum_i\la_i=r$.  Let $\La(N,r)$ denote the set of compositions of $r$ into $N$-parts. A partition $\pi$ of $r$ is a weakly decreasing sequence $(\pi_1,\pi_2,\cdots,\pi_t)$ of nonzero integers. Let $\Pi(r)$ denote the set of partitions of $r$.

The parabolic (or standard Young)
subgroup $\fS_{\lambda}$ of $\fS_r$ associated with a composition $\la$
consists of the permutations of $\{1,2,\cdots,r\}$ which leave
invariant the following sets of integers
$$R_1^\la=\{1,2,\cdots,\lambda_1\},R_2^\la=\{\lambda_1+1,\lambda_1+2,\cdots,\lambda_1+\lambda_2\},\cdots.$$
%$$R^\la_i=\{\tilde\lambda_{i-1}+1,\tilde\lambda_{i-1}+2,\ldots,\tilde\lambda_{i-1}+\lambda_i\}\quad(1\leq i\leq N).$$
%Here $\tilde\la_0=0$ and $\tilde\la_i=\sum_{j=1}^i\la_j$ so that $\tilde\la$ is the partial sum sequence associated with $\la$.

We will also denote by $\sD_\la:=\mathcal{D}_{\fS_\la}$ (resp., $\sD_\la^+$) the set of all
distinguished or shortest (resp., longest) coset representatives of the right
cosets of $\fS_\la$ in $\fS_r$.
Let
$\mathcal{D}_{\la\mu}=\mathcal{D}_\la\cap\mathcal{D}^{-1}_{\mu}$, where $\mu\in\La(N,r)$.
Then $\mathcal{D}_{\la\mu}$ (resp., $\mathcal{D}^+_{\la\mu}$) is the set of shortest (resp., longest)
$\fS_\la$-$\fS_\mu$ double coset representatives. For
$d\in\mathcal{D}_{\la\mu}$, the subgroup $\fS_\la^d\cap
\fS_\mu=d^{-1}\fS_\la d\cap \fS_\mu$ is a parabolic subgroup associated
with a composition which is denoted by $\la d\cap\mu$. In other
words, we define
\begin{equation}\label{ladmu}
\fS_{\la d\cap\mu}=\fS_\la^d\cap \fS_\mu.
\end{equation}
The composition $\la d\cap\mu$ can be easily described in terms of the following matrix.
Let
\begin{equation}\label{jmath}
\jmath(\la,d,\mu)=(a_{i,j}),\qquad\text{where }a_{i,j}=|R^\la_i\cap d(R^\mu_j)|,
\end{equation}
be the $N\times N$ matrix associated to the
 double coset $\fS_\la d \fS_\mu$. Then
 \begin{equation}\label{ladmu}
 \la d\cap\mu=(\nu^1,\nu^2,\ldots,\nu^N),
  \end{equation}
where $\nu^j=(a_{1,j},a_{2,j},\ldots,a_{N,j})$ is the $j$th column of $A$.
In this way, the matrix set
$$M(N,r)=\{\jmath(\la,d,\mu)\mid\la,\mu\in\La(N,r),d\in\sD_{\la\mu}\}$$
is the set of all $N\times N$ matrices over $\mathbb N$ whose entries sum to $r$.
For $A\in M(N,r)$, let
 $$\ro(A):=(\sum_{j=1}^Na_{1,j},\ldots,\sum_{j=1}^Na_{N,j})=\la\,\text{ and }\,\co(A):=(\sum_{i=1}^Na_{i,1},\ldots,\sum_{i=1}^Na_{i,N})=\mu.$$
%We also let row$_j(A)$ and col$_j(A)$ denote the $j$th row and $j$th column of $A$, respectively.

For nonnegative integers (not both zero) $m,n$, we often write a composition $\la=(\la_1,\ldots,\la_{m+n})\in\La(m+n,r)$ as $\la=(\lambda^{(0)}|\lambda^{(1)})$, where
 $$\lambda^{(0)}=(\lambda_1,\cdots,\lambda_m),\lambda^{(1)}=(\lambda_{m+1},\cdots,\lambda_{m+n}),$$
to indicate the``even''
and ``odd'' parts of $\la$ and identify $\Lambda(m+n,r)$ with the set
\begin{equation*}
\begin{aligned}
\Lambda(m|n,r)&=\{\lambda=(\lambda^{(0)}|\lambda^{(1)}) \mid\la\in\La(m+n,r)\}\\
&=\bigcup_{r_0+r_1=r}(\La(m,r_0)\times\La(n,r_1)).\\
\end{aligned}
\end{equation*}Let
\begin{equation*}
\begin{aligned}
\Lambda^+(m|n,r)&=\{\la\in\La(m|n,r)\mid \la_1\geq\cdots\geq\la_m,\la_{m+1}\geq\cdots\geq\la_{m+n}\},\\
\La(m|n)&=\bigcup_{r\geq0}\La(m|n,r)=\mathbb N^{m+n},\text{ and }\La^+(m|n)=\bigcup_{r\geq0}\La^+(m|n,r)\\
\end{aligned}
\end{equation*}
Thus, a parabolic subgroup $\fS_\la$ associated with  $\lambda=(\lambda^{(0)}|\lambda^{(1)})\in \Lambda(m,r_0)\times\Lambda(n,r_1)$  has the even part
$\fS_{(\la^{(0)}|1^{r_1})}$, briefly denoted by $\fS_{\lambda^{(0)}}$ , and the odd part $\fS_{(1^{r_0}\mid \la^{(1)})}$, denoted by $\fS_{\lambda^{(1)}}$.

 For
$\la,\mu\in\Lambda(m|n,r)$, let
\begin{equation}\label{Dcirc}
\mathcal{D}^\circ_{\la\mu}=\{d\in\mathcal{D}_{\la\mu}\mid
\fS^d_{\la^{(0)}}\cap \fS_{\mu^{(1)}}=1,\fS^d_{\la^{(1)}}\cap
\fS_{\mu^{(0)}}=1\}.
\end{equation}
This set is the super version of the usual $\sD_{\la\mu}$.
Let
\begin{equation}\label{M(m|n)}
\aligned
M(m|n,r)&=\{\jmath(\la,d,\mu)\mid\la,\mu\in\La(m|n,r),d\in\sD_{\la\mu}^\circ\}\text{ and }\\
M(m|n)&=\bigcup_{r\geq0}M(m|n,r).\endaligned
\end{equation}

Actually, from \cite[Prop.3.2]{DR}, if $(a_{i,j})\in M(m|n,r)$, then $a_{i,j}=0$ \mbox{ or } $1$ if $ i\leq m<j$ or $ j\leq m<i$.
We may extends the Bruhat order to $M(m|n,r)$ by setting, for $A=\jmath(\la,d,\mu),A'=\jmath(\la',d',\mu')\in M(m|n,r)$,
\begin{equation}\label{Bruhat}
A\leq A'\iff \la=\la', \mu=\mu', \text{ and }d\leq d'.
\end{equation}

Let $\mathcal{Z}=\mathbb Z[\up,\up^{-1}]$. The Hecke algebra $\mathcal{H}=\sH(\fS_r)$ associated to
$\fS=\fS_r$ is a free $\sZ$-module with basis $\{T_w; w\in
\fS_r\}$ and the multiplication is defined by the rules: for $s\in S$,
\begin{equation}
{T}_w{T}_s=\left\{\begin{aligned} &{T}_{ws},
&\mbox{if } l(ws)>l(w);\\
&(\up^2-1)T_w+\up^2{T}_{ws}, &\mbox{otherwise}.
\end{aligned}
\right.
\end{equation}
The bar involution on $\sH$ is the ring automorphism $\bar{\ }:\sH\to\sH$ defined by $\bar\up=\up^{-1}$ and $\bar T_w=(T_{w^{-1}})^{-1}$
for all $w\in\fS$.
%Let $\sH_\up(r)=\sH_\bsq(r)\otimes_\sA\sZ$.

%If $\fS'$ is a parabolic subgroup of $\fS_r$, then the $\sA$-module
%$\sum_{w\in W'}\sA\mathcal{T}_w$ is a subalgebra of $\mathcal{H}_\bsq(r)$,
%which is called a {\it parabolic subalgebra} of $\mathcal{H}_\bsq(r)$,
%denoted by $\mathcal{H}_{W'}$. We will use the abbreviation
%$\mathcal{H}_\lambda$ instead of $\mathcal{H}_{\fS_\lambda}$.

%Note that $(-1)^{\whd}(-1)^{\widehat{i_k}\widehat{i_{k+1}}}=(-1)^{\widehat{ds_k}}\text{
%for all } d\in\mathcal{D}_\lambda, s_k\in S\text{ with
%}ds_k\in\mathcal{D}_\lambda.$

For each $\lambda=(\lambda^{(0)}|\lambda^{(1)})\in \Lambda(m|n,r)$, define
$$x_{\lambda^{(0)}}=\sum_{w\in \fS_{\lambda^{(0)}}}T_w,y_{\lambda^{(1)}}=\sum_{w\in \fS_{\lambda^{(1)}}}(-\up^2)^{-l(w)}T_w.$$

\begin{definition}\label{S(m|n,r)} Let $\fT(m|n,r)=\bigoplus_{\lambda\in\Lambda(m|
n,r)}x_{\lambda^{(0)}}y_{\lambda^{(1)}}\mathcal{H}.$
The algebra
$$\sS(m|n,r):=
\End_{\mathcal{H}}(\fT(m|n,r))$$ is called a $\up$-{\it Schur superalgebra} over
$\sZ$ on which the $\mathbb Z_2$-graded structure is induced from the $\mathbb Z_2$-graded structure on $\fT(m|n,r)$ defined by
$$\fT(m|n,r)_i=\bigoplus_{{\lambda\in\Lambda(m|
n,r)}\atop {|\la^{(1)}|\equiv i(\text{mod}2)}}x_{\lambda^{(0)}}y_{\lambda^{(1)}}\mathcal{H}_{R} \;\;(i=\bar 0,\bar1).$$
\end{definition}
%Note that it is proved in \cite{DR} that $\sS(m|n,r)\cong\sS(m|n,r)\otimes_\sZ R$.

Following \cite{DR}, define, for $\lambda,\mu\in\Lambda(m|n,r)$ and
$d\in\mathcal{D}^\circ_{\lambda\mu}$,
$$T_{\fS_\lambda d \fS_\mu}:=\sum_{\substack{w_0w_1
\in \fS_\mu\cap \mathcal{D}_{\la d\cap\mu},\\
w_0\in \fS_{\mu^{(0)}},w_1\in \fS_{\mu^{(1)}}}}(-\up^2)^{-\ell(w_1)}x_{\lambda^{(0)}}y_{\lambda^{(1)}}T_dT_{w_0}T_{w_1}.$$
There exists $\mathcal{H}$-homomorphism
%$\psi^d_{\lambda\mu}$ such that
%$$(x_{\alpha^{(0)}}y_{\alpha^{(1)}}h)\psi_{\mu,\la}^d=\delta_{\mu,\alpha}T_{\fS_\lambda d
%W_\mu}h, \forall \alpha\in\Lambda(m|
%n,r),h\in\mathcal{H}.$$
%Note that we changed the left hand notation $\phi^d_{\lambda\mu}$ used in \cite[(5.7.1)]{DR}
%to the right hand notation $\psi^d_{\mu,\la}$, where
$$\phi_{\la\mu}^d(x_{\alpha^{(0)}}y_{\alpha^{(1)}})=\delta_{\mu,\alpha}T_{\fS_\lambda d
\fS_\mu}h, \forall \alpha\in\Lambda(m|
n,r),h\in\mathcal{H}.$$

If $A=\jmath(\lambda,d,\mu)$, denote $\phi_A:=\phi^d_{\lambda\mu}$.
The following result is given in {\cite[5.8]{DR}}.
\begin{lemma}\label{DR5.8} The set $\{\phi_A\mid
A\in M(m|n,r)\}$
forms a $\sZ$-basis for $\sS(m|n,r )$.
\end{lemma}

In order to define the canonical basis, we use the normalised basis $\{\varphi_A\mid
A\in M(m|n,r)\}$ defined as follow.
%For $w\in \fS_r$, set $\mathcal{T}_w=\up^{-l(w)}T_w$. If $D\subset \fS_r$, denote $\mathcal{T}_D=\sum_{w\in D}\mathcal{T}_w$.

For $\lambda,\mu\in \Lambda(m|n,r)$ and $d\in \mathcal{D}^\circ_{\lambda\mu}$, set $d^*$ (resp. $^*d$) to be the longest element in the double coset $\fS_{\lambda^{(0)}} d \fS_{\mu^{(0)}}$(resp. $\fS_{\lambda^{(1)}} d \fS_{\mu^{(1)}}$ ). If $A=\jmath(\lambda,d,\mu)$, by \cite[(6.0.2)]{DR}, let\footnote{We have corrected some typos given in \cite[(6.2.1)]{DR}.}
$$\mathcal{T}_A=\up^{-l(d^*)+l(^*d)-l(d)}T_A\;\;\text{ and }\;\;\varphi_A=\up^{-l(d^*)+l(^*d)-l(d)+l(w_{0,\mu^{(0)}})-l(w_{0,\mu^{(1)}})}\phi_A,$$
where $w_{0,\la}$ denotes the longest element in $\fS_\la$.

The {\it bar involution} on $\sH$ can be extended to the quantum Schur superalgebra
\begin{equation}\label{bar on S}
\bar{\ }:\sS(m|n,r)\longrightarrow \sS(m|n,r) \text{ satisfying }\bar\up=\up^{-1}, \overline{\vph_A}=\sum_{B\leq A}r_{B,A}\vph_B,
\end{equation}
where $r_{B,A}$ is defined by $\overline{\mathcal{T}_A}=\sum_{B\leq A}r_{B,A}\mathcal T_B$.

Let
\begin{equation}\label{hatA}
[A]=(-1)^{\wh A}\vph_A\quad\text{ where }\widehat A=\sum_{\substack{m<k<i\leq m+n\\1\leq j<l\leq m+n}}a_{i,j}a_{k,l}.
\end{equation}

Recall from \cite[\S8]{DR} that the $\up$-Schur superalgebra $\sS(m|n,r)$ can also be defined as the endomorphism algebra $\End_{\sH}(V(m|n)^{\otimes r})$ of the tensor space $V(m|n)^{\otimes r}$; see Corollary 8.4 there. Here $V(m|n)$ is a free $\sZ$-module of rank $m+n$ with basis $v_1,v_2,\cdots,v_{m+n}$, where $v_1,v_2,\cdots,v_m$ are even and $v_{m+1},\cdots,v_{m+n}$ are odd. Its tensor product $V(m|n)^{\otimes r}$ has the basis
$\{v_\bsi:=v_{i_1}\otimes v_{i_2}\otimes\cdots \otimes v_{i_r}\}_{\bsi\in I(m|n,r)}$ where
$$I(m|n,r)=\{\bsi=(i_1,i_2,\cdots,i_r)\mid 1\leq
i_j\leq m+n,\forall j\}.$$

The place
permutation (right) action of the symmetric group $\fS_r$ acts on $I(m|n, r)$ induces right $\sH$-module structure on $V(m|n)^{\otimes r}$; see \cite[(1.1.1)]{DG}.
For $A=\jmath(\lambda,d,\mu)\in M(m|n,r)$, let $\zeta_A\in\End_{\sH}(V(m|n)^{\otimes r})$ be defined by
$$\zeta_A(v_\mu)=(v_\mu) N_{\fS_r,\fS^d_\lambda\cap \fS_\mu}(e_{\mu,\lambda
d})=\sum_{w\in\mathcal{D}_{\la d\cap\mu}\cap\fS_\mu}(-\bsq)^{-l(w_1)}(v_{\bsi_\la d})T_w,$$
where $N_{\fS_r,\fS^d_\lambda\cap \fS_\mu}(e_{\mu,\lambda
d})$ is the relative norm defined in \cite[(1.1.2)]{DG},  $v_\mu=v_{\bsi_\mu}$ with
$$
\bsi_\mu=(\underbrace{1,\ldots,1}_{\mu_1},\underbrace{2,\ldots,2}_{\mu_2},\ldots,\underbrace{m+n,\ldots,m+n}_{\mu_{m+n}})=(1^{\la_1},2^{\la_2},\ldots,(m+n)^{\la_{m+n}})$$
and $w_1$ is an ``odd'' component of $w=w_0w_1$ with $w_i\in\fS_{\mu^{(i)}}$.
Following \cite[(4.2.1)]{DG}, let
\begin{equation}
\xib_A=\up^{-d(A)}\zeta_A\quad\text{where }d(A)=\sum_{i>k,j<l}a_{i,j}a_{k,l}+\sum_{j<l}(-1)^{\hat{i}}a_{i,j}a_{i,l}.
\end{equation}
 We have the following identification between the bases $\{[A]\}$ and $\{\xi_A\}$.

\begin{lemma}\label{DRB-DHB} By identifying $\sS(m|n,r)$ with $\End_{\sH}(V(m|n)^{\otimes r})$ under the isomorphism given in
\cite[Cor.~8.4]{DR}, we have $\xib_A=[A]=(-1)^{\wh A}\vph_A$ for all $A\in M(m|n,r)$.
\end{lemma}
\begin{proof} By \cite[Prop.8.3]{DR}, the map $f: V_R(m|n)^{\otimes r} \rightarrow\fT_R(m|n,r)$  sending $(-1)^{\hat{d}}v_{\bsi_\lambda d}$ to $x_{\lambda^{(0)}}y_{\lambda^{(1)}}T_d$, for any $\lambda\in\Lambda(m|n,r)$ and $d\in \mathcal{D}_\lambda$, is an $\mathcal{H}$-module isomorphism.  Here $\wh d=\sum_{k=1}^{r-1}\sum_{k<l,i_k>i_l}\widehat{i_k}\widehat{i_l}$ for $\bsi=\bsi_\la d$. It is direct to check that $\phi_A\circ f=(-1)^{\hat{d}}f\circ\zeta_A$

Now, for $A\in M(m|n,r)$ with $A=\jmath(\lambda, d,\mu)$, we have by Remark \cite[Remark 4.3]{DG},  $d(A)=l(d^*)-l(^*d)+l(d)-l(w_{0,\mu^{(0)}})+l(w_{0,\mu^{(1)}})$ and, by \cite[Lem.~2.3]{DG}, $\wh A=\wh d$.   The assertion follows immediately.
\end{proof}
In \cite{DR}, a canonical basis $\{\Theta_A\}_{A}$ is constructed relative to the basis $\{\vph_A\}_A$ and the bar involution defined in \cite[Th.~6.3]{DR}. By the lemma above, the canonical basis $\{\Xi_A\}_{A}$ relative to the basis $\{[A]\}_A$ and the same bar involution can be similarly defined.

\begin{corollary}\label{basis Xi} Let $\mathscr C_r=\{\Xi_A\mid A\in M(m|n,r)\}$ (resp., $\{\Theta_A\mid A\in M(m|n,r)\}$) be the canonical basis defined relative to basis $\{[A]\}_A$ (resp., $\{\vph_A\}_A$), the bar involution \eqref{bar on S}, and the Bruhat order $\leq$. Then $\Xi_A=(-1)^{\wh A}\Theta_A$.
\end{corollary}

\begin{proof}Since $\{\Xi_A\}_{A}$ (resp., $\{\Theta_A\}_{A}$) is the unique basis satisfying $\overline{\Xi_A}=\Xi_A$ (resp., $\overline{\Theta_A}=\Theta_A$) and
$$\Xi_A-[A]\in\sum_{B< A}\up^{-1}\mbz[\up^{-1}][B]\quad (\text{resp., }\Theta_A-\vph_A\in\sum_{B< A}\up^{-1}\mbz[\up^{-1}]\vph_B).$$
If we write $\Theta_A=\vph_A +\sum_{B< A}p_{B,A}\vph_B$, then, by the lemma above,
$$(-1)^{\wh A}\Theta_A=[A]+ \sum_{B< A}(-1)^{\wh A+\wh B}p_{B,A}[B]\text{ and }\overline{(-1)^{\wh A}\Theta_A}=(-1)^{\wh A}\Theta_A.$$
The uniqueness forces $\Xi_A=(-1)^{\wh A}\Theta_A$.
\end{proof}

We will discuss a PBW type basis for $\sS(m|n,r)$ at the end of \S5.

\section{Multiplication formulas and a stabilisation property}
We first record the following multiplication formulas discovered in \cite[Props.~4.4-5]{DG}.
For a fixed matrix $A\in M( m|n,r)$, $h\in[1,m+n)$ and $p\geq 1$, let
\begin{equation}\label{UL}
\aligned
U_p=U_p(h,\ro(A))&=\diag(\ro(A)-p\bse_{h+1})+pE_{h,h+1}\in M(m|n,r)\\
L_p=L_p(h,\ro(A))&=\diag(\ro(A)-p\bse_{h})+pE_{h+1,h}\in M(m|n,r).\endaligned
\end{equation}
\begin{proposition}\label{integral basis multiplication}Maintain the notation above. The
following multiplication formulas hold in the $\up$-Schur
superalgebra $\mathcal{S}(m|n,r)$ over $\mathcal{Z}$:

\noindent
If {\bf $\boldsymbol{h\neq m}$}, then
\begin{itemize}
\item[(1$^+$)] $[U_p][A]=\sum_{\substack{\nu\in\Lambda(m|n,p)\\\nu\leq
 \row_{h+1}(A)}}\up_h^{f_h(\nu,A)}\prod_{k=1}^{m+n}\overline{\left[\!\!\left[a_{h,k}+\nu_k\atop\nu_k\right]\!\!\right]}_{\up_h^2}
 [A+\sum_l\nu_l(E_{h,l}-E_{h+1,l})],$\\

\item[(1$^-$)]$[L_p][A]=\sum_{\substack{\nu\in\Lambda(m|n,p)\\\nu\leq
 \row_{h}(A)}}\up_{h+1}^{g_h(\nu,A)}\prod_{k=1}^{m+n}\overline{\left[\!\!\left[a_{h+1,k}+\nu_k\atop\nu_k\right]\!\!\right]}_{\up_{h+1}^2}
 [A-\sum_l\nu_l(E_{h,l}-E_{h+1,l})],$
\end{itemize}
 where $\la\leq\mu\iff \la_i\leq\mu_i\;\forall i$,
 \begin{equation}\label{beta_h1}
 \aligned
 f_h(\nu,A)&=\sum_{j\geq
 t}a_{h,j}\nu_t-\sum_{j>t}a_{h+1,j}\nu_t+\sum_{t<t'}\nu_t\nu_{t'},\text{ and }\\
 g_h(\nu,A)&=\sum_{j\leq
 t}a_{h+1,j}\nu_t-\sum_{j<t}a_{h,j}\nu_t+\sum_{t<t'}\nu_t\nu_{t'}.\endaligned
 \end{equation}
% Here $A+\sum_l\nu_l(E_{h,l}-E_{h+1,l}),A-\sum_l\nu_l(E_{h,l}-E_{h+1,l})\in M(m|n,r).$
\noindent
If {\bf$\boldsymbol{h= m}$}, then $[U_p][A]=0=[L_p][A]$ for all $p>1$ and
$$\aligned
(2^+)\; [U_1][A]&=\sum_{\substack{k\in[1,m+n]\\ a_{m+1,k}\geq
1}}(-1)^{\sum_{i>m,j<k}a_{i,j}}\up_m^{f_m(\bse_k,A)}\overline{[\![a_{m,k}+1]\!]}_{\up^2_m}[A+E_{m,k}-E_{m+1,k}];\\
(2^-)\; [L_1][A]&=\sum_{\substack{k\in[1,m+n]\\ a_{m,k}\geq
1}}(-1)^{\sum_{i>m,j<k}a_{i,j}}\up_{m+1}^{g_m(\bse_k,A)}\overline{[\![a_{m+1,k}+1]\!]}_{\up^2_{m+1}}[A-E_{m,k}+E_{m+1,k}],\endaligned$$
where
\begin{equation}\label{beta_m gamma_m}
%\aligned
f_m(\bse_k,A)=\sum_{j\geq k}a_{m,j}+\sum_{j>k}a_{m+1,j}\text{ and }
g_m(\bse_k,A)=\sum_{j\leq
k}a_{m+1,j}+\sum_{j<k}a_{m,j}.%\endaligned
\end{equation}
 \end{proposition}

We now describe a stabilisation property from these formulas which is the key to a realisation of the supergroup $\bfU(\glmn)$.

For $A=(a_{i,j})\in M(m|n,r)$, let
\begin{equation}\label{signAbar}
\bar{A}=\sum_{\substack{m+n\geq i> m\geq k\geq1
\\m<j<l\leq m+n}}a_{i,j}a_{k,l}.
\end{equation}
Note that this number is different from the number $\widehat A$ defined in \eqref{hatA}.

Consider the set of matrices with zero diagonal:
$$M(m|n)^\pm=\{A=(a_{i,j})\in M(m|n)\mid a_{i,i}=0,
1\leq i\leq m+n\}.$$
Define $M(m|n)^+$ (resp., $M(m|n)^-$) as the subset of upper (resp., lower) triangular elements in $M(m|n)^\pm$.
For $A\in M_{m+n}(\mbz)$ and
$\bsj=(j_1,j_2,\cdots,j_{m+n})\in\mathbb{Z}^{m+n}$, define
\begin{equation}\label{Ajr}
A(\bsj,r)=\begin{cases}\sum_{\substack{\lambda\in\Lambda(m|n,r-|A|)}}(-1)^{\overline{A+\diag(\lambda)}}\up^{\lambda\centerdot\bsj}[{A+\diag(\lambda})],&\text{ if }A\in M(m|n)^\pm,|A|\leq r;\\
0,&\text{ otherwise,}\end{cases}
\end{equation}
where $\centerdot=\centerdot_s$ denotes the super (or signed) ``dot product'':
\begin{equation}\label{dot product}
\lambda\centerdot\bsj=\sum_{i=1}^{m+n}(-1)^{\hat{i}}\lambda_ij_i=\la_1j_1+\cdots+\la_mj_m-\la_{m+1}j_{m+1}-\cdots-\la_{m+n}j_{m+n}.
\end{equation}

We have the following stabilisation property.
\begin{proposition}[{\cite[5.3, 5.6]{DG}}] \label{DGMF}
For all $r\geq0$, the set
$$\mathcal L_r=\{A(\bsj,r)\mid A\in M(m|n)^\pm,\bsj\in\mathbb{Z}^{m+n}\}$$ spans the $\up$-Schur superalgebra $\bsS(m|n,r)$ over $\mathbb Q(\up)$. Moreover, $E_{h,h+1}(\bfl,r)A(\bsj,r)$ and $E_{h+1,h}(\bfl,r)A(\bsj,r)$ can be written as a linear combination of certain (linearly independent) elements of $\mathcal L_r$ with coefficients independent of $r$ for all $r\geq|A|$.
\end{proposition}

We write explicitly a special case of these multiplication formulas.

\begin{lemma}[{\cite[6.2]{DG}}]\label{integral generators} For fixed $A=(a_{i,j})\in M(m|n)^+$ and $1\leq h<m+n$, let
\begin{equation}\label{sigmaf}
\aligned\sigma(k)&=\sigma_{A}(k):=\sum_{i\leq m,j>k}a_{i,j}\\
f(h,k)&=f_{A}(h,k):=\sum_{j\geq k}a_{h,j}-(-1)^{\delta_{m,h}}\sum_{j>k}a_{h+1,j},\endaligned \qquad(1\leq k\leq m+n).
\end{equation}
The following multiplication formulas hold in $\sS(m|n,r)$ for all $r\geq |A|$:
\begin{equation*}
\begin{aligned}
E_{h,h+1}(\mathbf{0},r)&A(\mathbf{0},r)=(-1)^{\sigma(h+1)\delta_{h,m}}\up_h^{f(h,h+1)}\overline{[\![a_{h,h+1}+1]\!]}_{\up^2_h}(A+E_{h,h+1})(\mathbf{0},r)\\
&\!\!\!\!\!+\sum_{k>h+1,a_{h+1,k}\geq
1}(-1)^{\sigma(k)\delta_{h,m}}\up_h^{f(h,k)}\overline{[\![a_{h,k}+1]\!]}_{\up^2_h}(A+E_{h,k}-E_{h+1,k})(\mathbf{0},r);\\
%(2)\;E_{h+1,h}(\mathbf{0},r)&A^-(\mathbf{0},r)=(-1)^{\vep\sigma^-(h)}\up_{h+1}^{f'(h)}\overline{[\![a_{h+1,h}+1]\!]}_{\up^2_{h+1}}(A^-+E_{h+1,h})(\mathbf{0},r)\\
%&\sum_{k<h,a_{h,k}\geq
%1}(-1)^{\vep\sigma^-(k)}\up_{h+1}^{f'(k)}\overline{[\![a_{h+1,k}+1]\!]}_{\up^2_{h+1}}(A^--E_{h,k}+E_{h+1,k})(\mathbf{0},r).
\end{aligned}
\end{equation*}
\end{lemma}

We now generalise this property to the higher order situation. By Proposition \ref{integral basis multiplication}(2$^\pm$), we only need to consider the $h\neq m$ case.

\begin{lemma}\label{pEA}
Let  $A=(a_{i,j})\in M(m|n)^+$ and $h\in[1,m+n]$ with $h\neq m$ and let $p$ be any positive integer. Then, for all $r\geq|A|$, the following multiplication formula holds in $\sS(m|n,r)$:
$$(pE_{h,h+1})({\bf0},r)A({\bf0},r)=\sum_{\nu\in\Lambda (m|n,p)\atop \nu-\nu_{h+1}{\bf e}_{h+1}\leq \row_{h+1}(A)}\up_h^{f_h(\nu,A)}\prod_{k=1}^{m+n}\overline{\left[\!\!\left[{a_{h,k}+\nu_k\atop \nu_k}\right]\!\!\right]}_{\up_h^2}A^{[\nu]}({\bf0},r),$$
where $f_h(\nu,A)$ is defined in \eqref{beta_h1} and
$A^{[\nu]}=A+\nu_{h+1}E_{h,h+1}+\sum_{l\neq h+1}\nu_l(E_{h,l}-E_{h+1,l}).$
%&(pE_{h+1,h})({\bf0},r)A({\bf0},r)\\
%&\!=\sum_{\nu\in\Lambda (m|n,p)}\up_{h+1}^{g_h(\nu,A)}\prod_{k=1}^{m+n}\overline{\left[\!\!\left[{a_{h+1,k}+\nu_k\atop \nu_k}\right]\!\!\right]}_{\up_{h+1}^2}(A+\nu_{h+1}E_{h,h+1}-\sum_{l\neq h+1}\nu_l(E_{h,l}-E_{h+1,l}))({\bf0},r),
%$$f_h(\nu,A)=\sum_{j\geq t}a_{h,j}\nu_t-\sum_{j>t}a_{h+1,j}\nu_t+\sum_{t<t'}\nu_t\nu_{t'}.$$
\end{lemma}

\begin{proof}
For notational simplicity, let
$$\La':=\Lambda(m|n,r-|A|)\text{ and }A^\mu=(a_{i,j}^\mu):=A+\diag(\mu)\;\;\forall \mu\in \La'.$$
By definition and Proposition \ref{integral basis multiplication}, the left hand side becomes
\begin{equation*}
\begin{aligned}
\text{LHS}&=\sum_{\lambda\in\Lambda(m|n,r-p)}[{pE_{h,h+1}+\diag(\lambda})]\sum_{\mu\in\La'}(-1)^{\overline{A^\mu}}[{A^\mu}]\\
&=\sum_{\mu\in\La'}(-1)^{\overline{A^\mu}}\sum_{\substack{\nu\in\Lambda(m|n,p)\\\nu\leq\row_{h+1}(A^\mu)}}\up_h^{f_h(\nu,A^\mu)}\prod_{k=1}^{m+n}\overline{\left[\!\!\left[{a^\mu_{h,k}+\nu_k\atop \nu_k}\right]\!\!\right]}_{\up_h^2}[{A^\mu+\sum_{l}\nu_l(E_{h,l}-E_{h+1,l})}]\\
&=\sum_{\mu\in \La'}\sum_{\substack{\nu\in\Lambda(m|n,p)\\\nu\leq\row_{h+1}(A^\mu)}}\up_h^{f_h(\nu,A)}\prod_{k=1}^{m+n}\overline{\left[\!\!\left[{a_{h,k}+\nu_k\atop \nu_k}\right]\!\!\right]}_{\up_h^2}(-1)^{\overline{A^\mu}}[{A^\mu+\sum_{l}\nu_l(E_{h,l}-E_{h+1,l})}],
\end{aligned}
\end{equation*}
where the last equality is seen as follows.  Since $A\in M(m|n)^+$,  the first $h$ entries of $\row_{h+1}(A^\mu)$ are zero. Hence,
%if $A+\mu+\sum_{l}\nu_l(E_{h,l}-E_{h+1,l})\in M(m|n,r)$,  then
$\nu=(\underbrace{0,\ldots,0}_h,\nu_{h+1},\ldots)$. Thus,
\begin{equation*}
\begin{aligned}
f_h(\nu,A^\mu)&=\sum_{j\geq t}a^\mu_{h,j}\nu_t-\sum_{j>t}a^\mu_{h+1,j}\nu_t+\sum_{t<t'}\nu_t\nu_{t'}\\
&=\sum_{j\geq t>h}a^\mu_{h,j}\nu_t-\sum_{j>t>h}a^\mu_{h+1,j}\nu_t+\sum_{t<t'}\nu_t\nu_{t'}\\
&=\sum_{j\geq t>h}a_{h,j}\nu_t-\sum_{j>t>h}a_{h+1,j}\nu_t+\sum_{t<t'}\nu_t\nu_{t'}=f_h(\nu,A)
\end{aligned}
\end{equation*}
 and
$$\prod_{k=1}^{m+n}\overline{\left[\!\!\left[{a^\mu_{h,k}+\nu_k\atop \nu_k}\right]\!\!\right]}_{\up_h^2}
=\prod_{k>h}^{m+n}\overline{\left[\!\!\left[{a_{h,k}+\nu_k\atop \nu_k}\right]\!\!\right]}_{\up_h^2}
=\prod_{k=1}^{m+n}\overline{\left[\!\!\left[{a_{h,k}+\nu_k\atop \nu_k}\right]\!\!\right]}_{\up_h^2}.$$

Let $\La''(\mu)=\{\nu\in\La(m|n,p)\mid \nu\leq\row_{h+1}(A^\mu)\}$. Then, for $p'=\min(p,r-|A|)$,
$$\sX:=\{(\mu,\nu)\mid\mu\in\La',\nu\in\La''(\mu)\}=\bigcup_{a=0}^{p'}\sX_a,$$
where the union is disjoint and
$$\sX_a=\{(\mu,\nu)\in\sX\mid\nu_{h+1}=a\}=\{(\mu,\nu)\mid \mu\in \La', \mu_{h+1}\geq a,\nu\in\La''(\mu),\nu_{h+1}=a\}.$$
Clearly there is a bijection between sets
$$\{\nu\in\La''(\mu)\mid \nu_{h+1}=a\}\text{ and }\{\nu'\in\La(m|n,p-a)\mid\nu'\leq\row_{h+1}(A)\},$$
where $\nu'=\nu-\nu_{h+1}\bse_{h+1}=(\underbrace{0,\ldots,0}_{h+1},\nu_{h+1},\ldots)$.
Moreover, since $h\neq m$, by \eqref{signAbar},
$$\overline{A^\mu}=\overline{A^\mu+\sum_l\nu_l(E_{h,l}-E_{h+1,l})}=\overline{A^{[\nu]}+\diag(\la-\mu_{h+1}\bse_{h+1})}.$$
Continuing our computation by swapping the summations yields
\begin{equation*}
\begin{aligned}
\text{LHS}
&=\sum_{a=0}^{p'}\sum_{(\mu,\nu)\in\sX_a}\up_h^{f_h(\nu,A)}\prod_{k=1}^{m+n}\overline{\left[\!\!\left[{a_{h,k}+\nu_k\atop \nu_k}\right]\!\!\right]}_{\up_h^2}(-1)^{\overline{A^\mu}}[{A^\mu+\sum_{l}\nu_l(E_{h,l}-E_{h+1,l})}]\\
&=\sum_{a=0}^{p'}\sum_{\mu\in\La',\atop \mu_{h+1}\geq a}\sum_{\nu'\in\Lambda(m|n,p-a)\atop \nu'\leq \row_{h+1}(A)}\up_h^{f_h(\nu,A)}\prod_{k=1}^{m+n}\overline{\left[\!\!\left[{a_{h,k}+\nu_k\atop \nu_k}\right]\!\!\right]}_{\up_h^2}(-1)^{\overline{A^\mu}}\\
&\quad\,\cdot[{A^\mu+a(E_{h,h+1}-E_{h+1,h+1})+\sum_{l\neq h+1}\nu_l(E_{h,l}-E_{h+1,l})}],
\end{aligned}
\end{equation*}
where $\nu=\nu'+a\bse_{h+1}$,
\begin{equation*}\begin{aligned}
\text{\qquad\;}&=\sum_{a=0}^{p'}\sum_{\nu'\in\Lambda(m|n,p-a),\atop\nu'\leq \row_{h+1}(A)}\up_h^{f_h(\nu,A)}\prod_{k=1}^{m+n}\overline{\left[\!\!\left[{a_{h,k}+\nu_k\atop \nu_k}\right]\!\!\right]}_{\up_h^2}\\
&\,\quad\cdot\sum_{\mu\in\La'\atop \mu_{h+1}\geq a}(-1)^{\overline{A^\mu}}[{A+aE_{h,h+1}+\sum_{l\neq h+1}\nu_l(E_{h,l}-E_{h+1,l})}+\diag(\mu-a\bse_{h+1})]\\
&=\sum_{\nu\in\Lambda (m|n,p)\atop \nu-\nu_{h+1}{\bf e}_{h+1}\leq \row_{h+1}(A)}\up_h^{f_h(\nu,A)}\prod_{k=1}^{m+n}\overline{\left[\!\!\left[{a_{h,k}+\nu_k\atop \nu_k}\right]\!\!\right]}_{\up_h^2}A^{[\nu]}({\bf0},r)=\text{RHS}.
\end{aligned}
\end{equation*}
\end{proof}

\section{The realisation of a PBW basis for $\bfU(\mathfrak{gl}_{m|n})$}

We now use the stabilisation property developed in Lemmas \ref{integral generators} and \ref{pEA} to give a realisation of the Lusztig form $U_\sZ^\pm(\mathfrak{gl}_{m|n})$ of the $\pm$-parts of
$\bfU(\mathfrak{gl}_{m|n})$ and introduce the canonical basis for $U_\sZ^\pm(\mathfrak{gl}_{m|n})$. We first recall the realisation of $\bfU(\mathfrak{gl}_{m|n})$ via the stabilisation property mentioned in Proposition \ref{DGMF}.

Recall also the definition of the super commutator on homogeneous elements of a superalgebra with parity function $\hat\  \,$:
$$[X,Y]=XY-(-1)^{\hat X\hat Y}YX.$$

\begin{definition}[\cite{Z}] \label{sgroup}The quantum supergroup
  $\bfU=\bfU(\mathfrak{gl}_{m|n})$ is the superalgebra over $\mathbb Q(\up)$ with
$$\aligned
\text{even generators:}\;&\sfK_a,\sfK_a^{-1}, \sfE_{h},\sfF_{h},\;1\leq a,h\leq m+n,h\neq m,m+n,\text{ and }\\
\text{odd generators:}\;&\sfE_{m}, \sfF_{m}\\
\endaligned$$
 which satisfy the following relations:
\begin{itemize}
\item[(QS1)] $\sfK_a\sfK_a^{-1}=1,\sfK_a\sfK_b=\sfK_b\sfK_a;$

\item[(QS2)] $\sfK_a
\sfE_{h}=\up^{\bse_a\centerdot_s\alpha_h}\sfE_{h}\sfK_a,
\sfK_a\sfF_{h}=\up^{-\bse_a\centerdot_s\alpha_h}\sfF_{h}\sfK_a;$

\item[(QS3)]
$[\sfE_{h},\sfF_{k}]=\delta_{h,k}\frac{\sfK_{h}\sfK_{h+1}^{-1}-\sfK_{h}^{-1}\sfK_{h+1}}{\up_h-\up_h^{-1}},$

\item[(QS4)] $\sfE_{h}\sfE_{k}=\sfE_{k}\sfE_{h},
\sfF_{h}\sfF_{k}=\sfF_{k}\sfF_{h},$ if $|k-h|>1;$

\item[(QS5)] For $h\neq m$ and $|h-k|=1$,
$$\aligned&\sfE^2_{h}\sfE_{k}-(\up_h+\up_h^{-1})\sfE_{h}\sfE_{k}\sfE_{h}+\sfE_{k}\sfE^2_{h}=0,\\
&\sfF_{h}^2\sfF_{k}-(\up_h+\up_h^{-1})\sfF_{h}\sfE_{k}\sfF_{h}+\sfF_{k}\sfF_{h}^2=0,
\endaligned$$
\item[(QS6)] $\sfE_{m}^2=\sfF_{m}^2=[\sfE_{m},\sfE_{m-1,m+2}]=[\sfF_{m},\sfE_{m+2,m-1}]=0$, where
$\sfE_{m-1,m+2}$, $\sfE_{m+2,m-1}$ denote respectively the elements
$$\aligned
&\sfE_{m-1}\sfE_{m}\sfE_{m+1}-\up\sfE_{m-1}\sfE_{m+1}\sfE_{m}-\up^{-1}\sfE_{m}\sfE_{m+1}\sfE_{m-1}+\sfE_{m+1}\sfE_{m}\sfE_{m-1},\\
&\sfF_{m+1}\sfF_{m}\sfF_{m-1}-\up^{-1}\sfF_{m}\sfF_{m+1}\sfF_{m-1}-\up\sfF_{m-1}\sfF_{m+1}\sfF_{m}+\sfF_{m-1}\sfF_{m}\sfF_{m+1}.
\endaligned$$
\end{itemize}
\end{definition}
Clearly, $\bfU$ admits a $\mbq(\up)$-algebra anti-involution (i.e., anti-automorphism of order two):
\begin{equation}\label{tau}
\tau:\bfU\longrightarrow\bfU,\quad \sfE_h\longmapsto\sfF_h,\sfF_h\longmapsto\sfE_h,\sfK_i^{\pm1}\longmapsto\sfK_i^{\pm1}.
\end{equation}

The {\it quantum root vectors} $\sfE_{a,b}$,
for $a,b\in[1,m+n]$ with $|a-b|\geq 1$, are defined by recursively setting $\sfE_{h,h+1}=\sfE_h,\sfE_{h+1,h}=\sfF_h,$  and
\begin{equation}\label{qbrackets}
\sfE_{a,b}=\begin{cases}
\sfE_{a,c}\sfE_{c,b}-\up^{-1}_c\sfE_{c,b}\sfE_{a,c}, &\mbox{  if } a<b;\\
\sfE_{a,c}\sfE_{c,b}-\up_c\sfE_{c,b}\sfE_{a,c}, &\mbox{  if } a>b,\\
\end{cases}
\end{equation}
where $c$ can be taken to be an arbitrary index strictly between $a$ and $b$, and $\sfE_{a,b}$ is homogeneous of degree $\hat E_{a,b}:=\hat{a}+\hat{b}$ . We remarks that $\tau$ does not send the positive root vector to negative root vectors, i.e., $\tau(\sfE_{a,b})\neq \sfE_{b,a}$ for all $a+1<b$.

Let $\Pi=\{\al_h=\bse_h-\bse_{h+1}\mid 1\leq h\leq m+n\}$. Then $\bfU$ has a natural grading over $\mbz\Pi$:
\begin{equation}\label{grading}
\bfU=\bigoplus_{\nu\in\mbz\Pi}\bfU_\nu
\end{equation}
such that $\sfK_i\in\bfU_0,\sfE_h\in\bfU_{\al_i}$ and $\sfF_h\in\bfU_{-\al_i}.$ We will write $\gd(x)=\nu$ if $x\in\bfU_\nu$, called the {\it graded degree} of $x$.

Consider the subspace $\bsfA(m|n)$ of the $\mbq(\up)$-algebra
$$\bsS(m|n):=\prod_{r\geq0}\bsS(m|n,r)$$ spanned by the linear independent set
$$\{A(\bsj):=\sum_{r\geq0} A(\bsj,r)\mid A\in M(m|n)^\pm,\bsj\in\mathbb{Z}^{m+n}\}.$$
By \cite[Ths.~9.1\&9.4]{DG} (deduced from Proposition~\ref{DGMF}), $\bsfA(m|n)$ is a subalgebra isomorphic to $\bfU(\mathfrak{gl}_{m|n})$. Moreover, there is an algebra isomorphism given by
\begin{equation}\label{eta}
\eta:\bfU(\mathfrak{gl}_{m|n})\longrightarrow\bsfA(m|n);
\sfE_h\mapsto E_{h,h+1}(\mathbf{0}),\;
\sfF_h\mapsto E_{h+1,h}(\mathbf{0}),\;\sfK_i^{\pm1}\mapsto O(\pm\bse_i).
\end{equation}

Let $U_\sZ^+=U_\sZ^+(\mathfrak{gl}_{m|n})$ be the $\sZ$-subalgebra of $\bfU(\mathfrak{gl}_{m|n})$ generated by all divided powers $\sfE_h^{(l)}:=\frac{\sfE_h^l}{[l]_{\up_h}^!}$, where $l\geq 1$ for all $h\neq m$. We have the following realisation of $U_\sZ^+$.

\begin{theorem} \label{A(0) basis}The $\sZ$-submodule $\fA_\sZ^+$ spanned by
$$\mathscr B=\{A(\bfl)\mid A\in M(m|n)^+\}$$ is a subalgebra of $\bsfA(m|n)$ which is isomorphic to $U_\sZ^+$. In other words, we have $\eta(U_\sZ^+)=\fA_\sZ^+$.
\end{theorem}
\begin{proof} The proof is somewhat standard; see e.g.,\cite[Th.~9.1]{DG}. Let $\fA_1^+$ be the $\sZ$-subalgebra generated by $(lE_{h,h+1})(\bfl)$ for all $l>0$ and $h\in[1,m+n)$ ($l=1$ if $h=m$). Then, by Lemmas \ref{integral generators} and \ref{pEA}, $\fA_1^+\subseteq\fA_\sZ^+$.
Further, by Lemma \ref{pEA}, the triangular relation \cite[(9.1.1)]{DG} can be taken over $\sZ$ (see \eqref{monomial} below). In particular, we can use this relation to prove that
$\fA^+_\sZ\subseteq\fA_1^+$.
\end{proof}

{\it We will identify $\bfU(\mathfrak{gl}_{m|n})$ and $U_\sZ^+$ with $\bsfA(m|n)$ and $\mathfrak A_\sZ^+$, respectively, under $\eta$ in the sequel.}

We now  take a closer look at the triangular relation mentioned in the proof. The order relation involved in the triangular relation is the following relation: for $A=(a_{i,j}),A'=(a'_{i,j})\in M(m|n)$,
 \begin{equation}\label{prec}
 A'\preceq A\iff
\begin{cases}
(1)\quad \sum_{i\leq s,j\geq t}a'_{i,j}\leq \sum_{i\leq
s,j\geq t}a_{i,j},&\text{for all $s<t$};\\
(2)\quad \sum_{i\geq s,j\leq t}a'_{i,j}\leq \sum_{i\geq
s,j\leq t}a_{i,j},&\text{for all $s>t$.}\end{cases}
\end{equation}
Note that this definition is independent of the diagonal entries of a matrix. So $\preceq$ is not a partial order on $M(m|n)$. However, its restriction to $M(m|n)^\pm$ is a partial order. In particular, we have posets $(M(m|n)^+,\preceq)$ and $(M(m|n)^-,\preceq)$.

Moreover, the following is taken from
\cite[Lem.~3.6(1)]{BLM} (see also \cite[Lem.~13.20,13.21]{DDPW}): for $A,B\in M(m|n,r)$,
\begin{equation}\label{prec order}
A\leq B\,(\text{the Bruhat order})\implies A\preceq B.
\end{equation}
We may also introduce another partial order $\preceq_{\rc}$ on $M(m|n,r)$ defined by\footnote{This order relation is denoted by $\sqsubseteq$ in \cite{BLM}.}
\begin{equation}\label{prec rc}
X\preceq_{\rc} Y\iff \ro(X)=\ro(Y), \co(X)=\co(Y),\text{ and }X\preceq Y.
\end{equation}
%This order will be generalised to matrices with negative diagonal entries and will be used to define canonical basis for the modified quantum group.

\begin{remark} \label{Xi relative to order rc}
Since $X\leq Y\implies X\preceq_{\rc} Y\implies X\preceq Y$, the canonical bases $\{\Xi_A\mid A\in M(m|n,r)\}$ defined in Corollary \ref{basis Xi} can also be defined relative to the basis $\{[A]\}_A$, the bar involution and the order $\preceq_\rc$.
\end{remark}

For any $A=(a_{i,j})\in M(m|n)^\pm$ and $\bsj\in\mathbb{Z}^{m+n}$, we have the following triangular relation in the $\mbq(\up)$-algebra
$\bsfA(m|n)$
$$\prod^{(\leq_2)}_{i,h,j}(a_{j,i}E_{h+1,h})(\mathbf{0})\cdot \prod^{(\leq_1)}_{i,h,j}(a_{i,j}E_{h,h+1})(\mathbf{0})=A(\bfl)+\sum_{\substack{B\in M(m|n)^\pm,\bsj\in\mbz^{m+n}\\B\prec A}}g_{B,A,\bsj}B(\bsj),$$
where $i,h,j$ satisfy $1\leq i\leq h<j\leq m+n$ and the products follow the orders $\leq_i$ which are defined as in \cite[(13.7.1)]{DDPW}. In particular, by Lemma \ref{pEA}, a single product for $A\in M(m|n)^+$ can be simplified as
\begin{equation}\label{monomial}
{\fm}^{+}_A:= \prod^{(\leq_1)}_{1\leq i\leq h<j\leq m+n}(a_{i,j}E_{h,h+1})(\mathbf{0})=A({\bf0})+\sum_{B\in M(m|n)^+,B\prec A}g_{B,A}B({\bf0}),
\end{equation}
where $g_{B,A}\in\sZ$. Then, applying the anti-involution $\tau$ in \eqref{tau} yields
\begin{equation}\label{mon-}
{\fm}^{-}_{A^t}:=\tau({\fm}^{+}_A)=\prod^{(\leq_1^{\text{op}})}_{i,h,j}(a_{i,j}E_{h+1,h})(\mathbf{0}).
\end{equation}

\begin{corollary}\label{monomial basis}
The set $\{\fm^{+}_A\mid A\in M(m|n)^+\}$ (resp., $\{\fm^{-}_A\mid A\in M(m|n)^-\}$) forms a $\sZ$-basis, {\sf a monomial basis}, for $U^+_\sZ$ (resp., $U^-_\sZ$).
\end{corollary}

We end this section by showing that the basis $\mathscr B=\{A(\bfl)\mid A\in M(m|n)^+\}$ identifies a PBW basis for $U^+_\sZ$.

For a root vector $\sfE_{k,l}$ with $k<l$ and $p>0$, if $\sfE_{k,l}^{p}\neq 0$  define the usual divided powers  $\sfE_{k,l}^{(p)}=\frac{\sfE_{k,l}^p}{[p]_{\up_k}^!}$. If we order linearly the set
$$\mathcal{J}'=\{(i,j)\mid 1\leq i<j\leq m+n\}$$ by setting,
for $(i,j),(i',j')\in \mathcal{J}'$, $(i,j)<_3 (i',j')$ if and only if
$j> j'$ or $j=j',i>i' $, and use the order to define, for any $A=(a_{i,j})\in M(m|n)^+$, the product and its `transpose'
\begin{equation}\label{E_A}
\sfE_{A}=\prod_{(i,j)\in \mathcal{J}'}^{(\leq_3)}\sfE_{i,j}^{(a_{i,j})}\;\text{ and }\;\sfF_{A^t}=\tau(\sfE_A).
\end{equation}
then the set $\{\sfE_{A}\mid A\in M(m|n)^+\}$ (resp. $\{\sfF_{A}\mid A\in M(m|n)^-\}$) forms a  PBW basis of $U_\sZ^+$ (resp., $U_\sZ^-$). We now prove that this basis is nothing but the same basis given in Theorem \ref{A(0) basis}.

For any $A\in M(m|n)$, set $$\|A\|=\sum_{1\leq i<j\leq m+n}\frac{(j-i)(j-i+1)}{2}(a_{i,j}+a_{j,i}).$$
%In particular, if $A\in M(m|n)^+$, then $\|A\|=\sum_{1\leq i<j\leq m+n}\frac{(j-i)(j-i+1)}{2}(a_{i,j}).$
Refer to \cite[Lem. 13.21]{DDPW}, for $A,B\in M(m|n)$,
$$B\prec A\implies \|B\|<\|A\| \;(\mbox{and } B\preceq A\implies \|B\|\leq \|A\|).$$

\begin{theorem}\label{PBW basis}
 For any $A\in M(m|n)^+$, we have
 $\sfE_{A}=A({\bf0})$. In other words, with the isomorphism $\eta$ given in \eqref{eta}, we have $\eta(\sfE_A)=A(\bfl).$
\end{theorem}

\begin{proof} Let $A=(a_{i,j})$.
We apply induction on $\|A\|$ to prove the assertion. If $\|A\|=1$, then $A$ must be of the form $E_{i,i+1}$ for some
$1\leq i<m+n$. Thus, this case is clear from the definition of $\eta$. So
$\sfE_{i,i+1}=(E_{i,i+1})({\bf0}),$ as desired.

Assume now $\|A\|>1$ and that, for any $B\in M(m|n)^+$ with $\|B\|<\|A\|$, $\sfE_{B}=B({\bf0})$.
Consider the entries of $A$ and choose $1\leq h<l\leq m+n$ such that $a_{h,l}>0$ and $a_{i,j}=0$ for all $j>l$ or $i>h$ whenever $j=l$. In other words, $\sfE_{h,l}^{(a_{h,l})}$ is the first factor in the product $\sfE_A$. Then, by the definition,
$$\sfE_{A}=\frac{1}{[a_{h,l}]_{\up_h}}\sfE_{h,l}\sfE_{A-E_{h,l}}.$$
Since $A-E_{h,l}\prec A$, $\|A-E_{h,l}\|<\|A\|$. By induction, we have $\sfE_{A-E_{h,l}}=(A-E_{h,l})({\bf0})$ and also $\sfE_{h,l}=E_{h,l}(\bfl)$.
%Thus, $$\sfE_{A}=\frac{1}{[a_{h,l}]_{\up_h}}E_{h,l}({\bf0})(A-E_{h,l})({\bf0}).$$
There are two cases to consider.

{\bf Case 1: $l=h+1$.} For this case, we directly use the multiplication formula given in Lemma \ref{integral generators}. By the selection of indices $h,l$, all $a_{h+1,j}=0=a_{h,j}$ if $j>h+1=l$. Thus, by \eqref{sigmaf}, $f_{A-E_{h,h+1}}(h,h+1)=a_{h,h+1}-1$. So
\begin{equation*}
\sfE_{h,h+1}\sfE_{A-E_{h,h+1}}=(E_{h,h+1})({\bf0})(A-E_{h,h+1})({\bf0})=\up^{a_{h,h+1}-1}_h\overline{[\![a_{h,h+1}]\!]}_{\up_h}A({\bf0}).
\end{equation*}
But then
$\up^{a_{h,h+1}-1}_h\overline{[\![a_{h,h+1}]\!]}_{\up_h}=[a_{h,h+1}]_{\up_h}$. Hence,
$\sfE_{A}=A({\bf0})$, as desired.

{\bf Case 2: $l>h+1$.} In this case, write $\sfE_{h,l}=\sfE_{h,h+1}\sfE_{h+1,l}-\up^{-1}_{h+1}\sfE_{h+1,l}\sfE_{h,h+1}$.
Since $A-E_{h,l}+E_{h+1,l}\prec A$, by induction,
\begin{equation}\label{EEA0}
\sfE_{h+1,l}\sfE_{A-E_{h,l}}=\sfE_{A-E_{h,l}+E_{h+1,l}}=(A-E_{h,l}+E_{h+1,l})({\bf0})
\end{equation}
and, on the other hand, %for $\vep=\delta_{h,m}$
\begin{equation}\label{EEA2}
\begin{split}
\sfE_{h,h+1}&\sfE_{A-E_{h,l}}=(E_{h,h+1})({\bf0})(A-E_{h,l})({\bf0})\\
&=(-1)^{\sigma_{A-E_{h,l}}(h+1)\delta_{h,m}}\up_h^{f_{A-E_{h,l}}(h,h+1)}\overline{[\![a_{h,h+1}+1]\!]}_{\up_h}(A-E_{h,l}+E_{h,h+1})({\bf0})\\
&+\sum_{a_{h+1,j}\geq 1}(-1)^{\sigma_{A-E_{h,l}}(j)\delta_{h,m}}\up_h^{f_{A-E_{h,l}}(h,j)}\overline{[\![a_{h,j}+1]\!]}_{\up_h}(A-E_{h,l}+E_{h,j}-E_{h+1,j})({\bf0}).
\end{split}
\end{equation}

Now multiplying \eqref{EEA0} by $\sfE_{h,h+1}$ and applying Lemma \ref{integral generators} yields
\begin{equation}\label{EEA1}
\begin{aligned}
\sfE_{h,h+1}&\sfE_{h+1,l}\sfE_{A-E_{h,l}}=(E_{h,h+1})({\bf0})(A-E_{h,l}+E_{h+1,l})({\bf0})\\
&=(-1)^{\sigma(h+1)\delta_{h,m}}\up_{h}^{f(h,h+1)}\overline{[\![a_{h,h+1}+1]\!]}_{\up_h}(A-E_{h,l}+E_{h+1,l}+E_{h,h+1})({\bf0})\\
&+\sum_{a_{h+1,j}\geq 1,j\neq l}(-1)^{\sigma(j)\delta_{h,m}}\up_h^{f(h,j)}\overline{[\![a_{h,j}+1]\!]}_{\up_h}(A-E_{h,l}+E_{h+1,l}+E_{h,j}-E_{h+1,j})({\bf0})\\
&+\up_h^{a_{h,l}-1}\overline{[\![a_{h,l}]\!]}_{\up_h}A({\bf0}),\\
\end{aligned}
\end{equation}
where $\sigma(k)=\sigma_{A-E_{h,l}+E_{h+1,l}}(k)=\sigma_A(k)-\delta_{h,m}$ and $f(h,k)=f_{A-E_{h,l}+E_{h+1,l}}(h,k)$.

We now compute multiplying \eqref{EEA2} by $\sfE_{h+1,l}$.
Note that, since $a_{h+1,j}=0$ for all $j\geq l$, the summation in \eqref{EEA2} is taken over all $j$ with $a_{h+1,j}\geq 1$ and $j<l$.
But $j<l$ implies
$$A-E_{h,l}+E_{h,h+1}\prec A\text{ and }A-E_{h,l}+E_{h,j}-E_{h+1,j}\prec A.$$
Thus, by induction,
$$\aligned
\sfE_{A-E_{h,l}+E_{h,h+1}}&=(A-E_{h,l}+E_{h,h+1})({\bf0})\\
\sfE_{A-E_{h,l}+E_{h,j}-E_{h+1,j}}&=(A-E_{h,l}+E_{h,j}-E_{h+1,j})({\bf0}).\endaligned
$$
Hence, by induction again,
$$\aligned
\sfE_{h+1,l}\sfE_{A-E_{h,l}+E_{h,h+1}}&=\sfE_{A-E_{h,l}+E_{h,h+1}+E_{h+1,l}}=(A-E_{h,l}+E_{h,h+1}+E_{h+1,l})({\bf0})\\
\sfE_{h+1,l}\sfE_{A-E_{h,l}+E_{h,j}-E_{h+1,j}}&=\sfE_{A-E_{h,l}+E_{h,j}-E_{h+1,j}+E_{h+1,l}}\\
&=(A-E_{h,l}+E_{h,j}-E_{h+1,j}+E_{h+1,l})({\bf0}),\\
\endaligned$$
since $A-E_{h,l}+E_{h,h+1}+E_{h+1,l}\prec A$ and $A-E_{h,l}+E_{h,j}-E_{h+1,j}+E_{h+1,l}\prec A.$

Thus, for $\vep=\delta_{h,m}$,
\begin{equation}\label{EEA3}
\begin{aligned}
&\sfE_{h+1,l}\sfE_{h,h+1}\sfE_{A-E_{h,l}}=\sfE_{h+1,l}\cdot(\text{RHS of \eqref{EEA2}})\\
&=(-1)^{\vep\sigma_{A-E_{h,l}}(h+1)}\up_h^{f_{A-E_{h,l}}(h,h+1)}\overline{[\![a_{h,h+1}+1]\!]}_{\up_h}(A-E_{h,l}+E_{h,h+1}+E_{h+1,l})({\bf0})\\
&+\sum_{a_{h+1,j}\geq 1}(-1)^{\vep(\sigma_{A-E_{h,l}}(j))}\up_h^{f_{A-E_{h,l}}(h,j)}\overline{[\![a_{h,j}+1]\!]}_{\up_h}(A-E_{h,l}+E_{h,j}-E_{h+1,j}+E_{h+1,l})({\bf0}).
\end{aligned}
\end{equation}
Finally, since $$f(h,j)=f_{A-E_{h,l}}(h,j)-(-1)^\vep,$$ and $\sigma(j)-\vep=\sigma_{A-E_{h,l}}(j)$ for $h=m$,
Combining \eqref{EEA1} and \eqref{EEA3} gives
\begin{equation*}
\sfE_{h,l}\sfE_{A-E_{h,l}}=\sfE_{h,h+1}\sfE_{h+1,l}\sfE_{A-E_{h,l}}-\up^{-1}_{h+1}\sfE_{h+1,l}\sfE_{h,h+1}\sfE_{A-E_{h,l}}
=[a_{h,l}]_{\up_h}A({\bf0}),
\end{equation*}
proving $\sfE_{A}=A({\bf0})$.
\end{proof}

\section{Canonical basis for $U_\sZ^\pm(\mathfrak{gl}_{m|n})$}
We are now ready to introduce the canonical basis for $U^\pm_\sZ$ via the PBW basis described in \eqref{E_A} and the partial order $\preceq$ used in the triangular relation \eqref{monomial}. We need another ingredient---a bar involution.

By Definition \ref{sgroup}, we may define the bar involution %(a ring automorphism of order 2)
\begin{equation}\label{bar on U}
\bar{\ }:\bfU(\mathfrak{gl}_{m|n})\rightarrow\bfU(\mathfrak{gl}_{m|n})\text{ with }\bar\up=\up^{-1},\bar\sfE_h=\sfE_h,\bar\sfF_h=\sfF_h, \bar\sfK_i^\pm=\sfK_i^\mp.
\end{equation}

\begin{remark} \label{bar compatibility}
(1) If we denote $\flat$ to be the involution on the direct product $\bsS(m|n)=\prod_{r\geq0}\bsS(m|n,r)$ defined by baring on every component (see \eqref{bar on S}), then the restriction of $\flat$ to
$\bsfA(m|n)=\bfU(\mathfrak{gl}_{m|n})$ coincides with the bar involution on $\bfU(\mathfrak{gl}_{m|n})$. This can be seen as follows.

If $A=\diag(\lambda)$ or $\diag(\la)+E_{h,h+1}$, then $A$ is minimal under the Bruhat ordering. Thus, \eqref{bar on S} implies $\overline{[A]}=[A]$. Since $E_{h,h+1}({\bf0},r)=\sum_{\lambda\in \Lambda(m|n,r-1)}[{E_{h,h+1}+\diag(\la)}]$ and $O(\bse_i,r)=\sum_{\la\in\La(m|n,r)}\up_i^{\la_i}[\diag(\la)]$, it follows that $\overline{E_{h,h+1}({\bf0},r)}=E_{h,h+1}({\bf0},r)$ and $\overline{O(\bse_i,r)}=O(-\bse_i,r)$.
Similarly, $\overline{E_{h+1,h}({\bf0},r)}=E_{h+1,h}({\bf0},r)$. Hence, $\flat(E_{h,h+1}({\bf0}))=E_{h,h+1}({\bf0})$, $\flat(O(\bse_i))=O(-\bse_i)$, and $\flat(E_{h+1,h}({\bf0}))=E_{h+1,h}({\bf0})$. This is the same bar involution as defined in \eqref{bar on U}.

(2) Matrix transposing induces $\mbq(\up)$-algebra anti-involutions
$$\tau_r:\bsS(m|n,r)\longrightarrow \bsS(m|n,r),\quad \xi_A\longmapsto\xi_{A^t},$$
which induce anti-involution $\tau$ on $\bsS(m|n)$ and, hence, on $\bsfA(m|n)=\bfU(\mathfrak{gl}_{m|n})$. This is the same $\tau$ as defined in \eqref{tau}. Hence, $\tau(A(\bfl))=A^t(\bfl)$ for all $A\in M(m|n)^+$. In particular, by \eqref{E_A}, we have $\sfF_A=A(\bfl)$ for all $A\in M(m|n)^-$.

(3) Let $\iota=\flat\circ\tau=\tau\circ\flat$. Then $\iota$ is a ring anti-involution
\begin{equation}\label{iota}
\iota:\bfU\longrightarrow\bfU,\quad \sfE_h\longmapsto\sfF_h,\sfF_h\longmapsto\sfE_h,\sfK_i^{\pm1}\longmapsto\sfK_i^{\mp1},\up^{\pm1}\longmapsto\up^{\mp1}.
\end{equation}
Note that $\iota(\sfE_{a,b})=\sfE_{b,a}$ for all $a\neq b$ (cf. the remark after \eqref{qbrackets}).
\end{remark}

\begin{theorem}\label{canonical basis}The basis $\{A(\bfl)\}_{A\in M(m|n)^+}$, the bar involution, and the partial order $\preceq$ define uniquely the canonical $\mathscr{C}^+=\{\cC_A\mid A\in M(m|n)^+\}$ for $U^+_\sZ$. In other words, the elements $\cC_A$ are uniquely defined by the conditions $\bar \cC_A=\cC_A$ and
$$\cC_A-A(\bfl)\in\sum_{B\in M(m|n)^+,B\prec A}\up^{-1}\mbz[\up^{-1}]B(\bfl).$$

Applying the anti-involution $\tau$ yields the canonical basis
\begin{equation}\label{C-}
\mathscr C^-=\{\cC_A=\tau(\cC_{A^t})\mid A\in M(m|n)^-\}
\end{equation}
 of $U^-_\sZ$ which can be defined similarly relative to $\{A(\bfl)\}_{A\in M(m|n)^-}$ etc.
\end{theorem}
\begin{proof} By \eqref{monomial}, we may write the basis $\{A(\bfl)\}_{A\in M(m|n)^+}$ in terms of the monomial basis:
$$A({\bf0})=\fm_A^++\sum_{B\in M(m|n)^+,B\prec A}h_{B,A}\fm_B^+,$$
where $h_{B,A}\in \sZ$. Applying the bar involution and \eqref{monomial} yields
$$\aligned
\overline{A(\bfl)}&=\fm_A^++\sum_{B\in M(m|n)^+,B\prec A}\overline{h_{B,A}}\,\fm_B^+\\
&=A(\bfl)+\sum_{A'\in M(m|n)^+,A'\prec A}f_{A',A}\,A'(\bfl)\quad(f_{A',A}\in\sZ).
\endaligned$$
Thus, a standard construction (see, e.g., \cite[\S0.5]{DDPW}) shows that there exist polynomials $p_{B,A}$ with $p_{A,A}=1$ and $p_{B,A}\in\up^{-1}\mbz[\up^{-1}]$ if $B\prec A$ such that  the elements
\begin{equation}\label{C_A}
\cC_A=A({\bf0})+\sum_{B\in M(m|n)^+,B\prec A}p_{B,A}B({\bf0})\quad (A\in M(m|n)^+)
\end{equation}
form the required basis. For the last assertion, see \eqref{mon-} and Remark \ref{bar compatibility}(2).\end{proof}

Recall from \cite[Cor.6.4]{DG} that there are $\mathbb Q(\up)$-superalgebra epimorphisms
\begin{equation}
\eta_r:\bfU(\mathfrak{gl}_{m|n})\longrightarrow \bsS(m|n,r)
\end{equation}
sending $\sfE_h,\sfF_h$ and $\sfK^{\pm1}$ to $E_{h,h+1}({\bf0},r),E_{h+1,h}({\bf0},r)$ and $O(\pm\bse_i,r)$, respectively.
Note that it was these epimorphisms that induce the isomorphism $\eta$ in \eqref{eta}.
Note also that the epimorphism $\eta_r$  is compatible with the bar involutions by the remark above.

Let $ \sS^-(m|n,r)=\eta_r({U_\sZ^-})$ and $ \sS^+(m|n,r)=\eta_r({U_\sZ^+})$. The subalgebra $ \sS^-(m|n,r)$ (resp. $ \sS^+(m|n,r)$) has a basis $$\mathscr B^-_r=\{A({\bf0},r)\mid A\in M(m|n)^-_{\leq r}\} \quad(\mbox{resp. }\mathscr B^+_r= \{A({\bf0},r)\mid A\in M(m|n)^+_{\leq r}\}),$$
where $M(m|n)^*_{\leq r}=\{A\in M(m|n)^*,|A|\leq r\}$ for $*=+,-$.

\begin{corollary}\label{image of canonical basis} For any $r>0$, let $\sfc_A=\eta_r(\cC_A)$ for all $A\in M(m|n)^+_{\leq r}$. Then
 $\{\sfc_A\}_{A\in M(m|n)_{\leq r}^+}$ forms a basis for $ \sS^+(m|n,r)$ which satisfies the following properties:
$$\overline{\sfc_A}=\sfc_A \mbox{ and } \sfc_A-A({\bf0},r)\in \sum_{B\prec A}\up^{-1}\mathbb{Z}[\up^{-1}]B({\bf0},r).$$
In other words, this is the canonical basis relative to $\mathscr B_r^+$ and the restrictions of the bar involution \eqref{bar on S} and $\preceq$.
Moreover, we have
\begin{equation}
\eta_r(\cC_A)=\left\{
\begin{aligned}
&\sfc_A, \mbox{ if } A\in M(m|n)_{\leq r}^+,\\
&0,\mbox{ otherwise}.
\end{aligned}
\right.
\end{equation}
A similar result holds for $ \sS^-(m|n,r)$.
\end{corollary}

%The proof of this proposition is identical to the non-super case in \cite[Lem. 8.2]{DP}, so we omit it.
\begin{proof}Applying $\eta_r$ to \eqref{monomial} yields a triangular relation in $\sS^+(m|n,r)$ between the monomial basis and $\mathscr B^+_r$. So the canonical basis, defined by the basis $\mathscr B^+_r$, the bar involution and the order $\preceq$, exists.
By Remark \ref{bar compatibility} and \eqref{C_A}, $\sfc_A$ does satisfy the described conditions. Hence, it is the required canonical basis.
\end{proof}

We now make a comparison between the canonical bases $\mathscr C^-$ for $U_\sZ^-$ and $\mathscr C_r$ for the quantum Schur superalgebra given in Corollary \ref{basis Xi}. Note that $\mathscr C_r$ can be defined by using the order $\preceq_\rc$ (see Remark \ref{Xi relative to order rc}).

Let $U_\sZ^0$ be the $\sZ$ subalgebra of $\bfU(\mathfrak{gl}_{m|n})$ generated by $\sfK_i$ and $\big[{\sfK_i;0\atop t}\big]$ for $1\leq i\leq m+n$ and $t\geq1$, where\footnote{We have corrected typos in line 3 from \cite[Th.~8.4]{DG}.}
$$\bigg[{\sfK_i;0\atop t}\bigg]=\prod_{a=1}^t\frac{\sfK_i\up_i^{-a+1}-\sfK_i^{-1}\up_i^{a-1}}{\up_i^{a}-\up_i^{-a}}.$$
Since, for any $\la\in\La(m|n,r)$, $\eta_r\big(\prod_{i=1}^{m+n}\big[{\sfK_i;0\atop \la_i}\big]\big)=[\diag(\la)]$, the $\sZ$-subalgebra
$\sS^{\geq0}(m|n,r):=\eta_r(U_\sZ^0U_\sZ^+)$ is spanned by $\{[\diag(\mu)]A(\bfl,r)\mid A\in M(m|n,r)^+,\mu\in\La(m|n,r)\}$. This is called a Borel subsuperalgebra.

For $A\in M(m|n)$, define the ``hook sums''
$$\fh_i(A)=a_{i,i}+\sum_{i<j}(a_{i,j}+a_{j,i})\text{ and } \bsfh(A)=(\fh_1(A),\ldots,\fh_{m+n}(A)).$$
If we write $A=A^-+A^0+A^+$, where $A^\pm\in M(m|n)^\pm$ and $A^0$ is diagonal, then
$$\bsfh(A)=(a_{1,1},\cdots,a_{n,n})+\co(A^-)+\ro(A^+).$$
Set $\bsfh(A)\leq\la$ if and only if $\fh_i(A)\leq \lambda_i$ for every $i$.

For $A\in M(m|n)^-$ and $\la\in\La(m|n,r)$ with $\bsfh(A)=\co(A)\leq\la$, let
\begin{equation}\label{Ala}
A_\lambda:=A+\diag(\la-\bsfh(A)).
\end{equation}
 It is clear that $[A_\la]=A(\bfl,r)[\diag(\la)]$ and the set
$$\{[A_\la]\mid A\in M(m|n,r)^-,\la\in\La(m|n,r),\bsfh(A)\leq\la\}$$
forms  a basis for $\sS^{\leq0}(m|n,r)$.

We have the following theorem (cf.  \cite[Th. 8.3]{DP}).

\begin{theorem}\label{canonical basis of two parts} For any $A\in M(m|n)^-$, if $|A|\leq r$, then $\sfc_A=\sum_{\bsfh(A)\leq\la}(-1)^{\bar A_\la}\Xi_{A_\la}$.
In other words, the image $\sfc_A$ of the canonical basis element $\cC_A\in U_\sZ^-$ under $\eta_r$ is either zero or a sum of the canonical basis elements $\Xi_{A_\la}=\sfc_A[\diag(\la)]\in\sS(m|n,r)$ ($\bsfh(A)\leq \la$). Moreover,
$$\mathscr{C}_r^{\leq 0}=\{\Xi_{A_\la}\mid A\in M(m|n,\leq r)^-,\lambda\in\Lambda(m|n,r),\sigma_i(A)\leq \lambda_i\}=\mathscr C_r\cap\sS^{\leq 0}(m|n,r)$$
forms the canonical basis for the Borel subsuperalgebra $\sS^{\leq0}(m|n,r)$.

A similar result holds for $\sS^{\geq0}(m|n,r)$
\end{theorem}

\begin{proof}
By Remark \ref{bar compatibility} and Corollary \ref{image of canonical basis}, we have $\overline{\sfc_A}=\sfc_A$ and  $\overline{[\diag(\la)]}=[\diag(\la)]$.  Hence, $\overline{\sfc_A[\diag(\lambda)]}=\sfc_A[\diag(\lambda)]$.

Since, by definition,
$\sfc_A=A({\bf0},r)+\sum_{B\prec A} p_{B,A} B({\bf0},r)$, where $p_{B,A}\in \up^{-1}\mathbb{Z}[\up^{-1}]$, it follows that
$$\aligned
\sfc_A[\diag(\la)]&=A({\bf0},r)[\diag(\la)]+\sum_{B\prec A} p_{B,A} B({\bf0},r)[\diag(\la)]\\
&=(-1)^{\overline{A_\lambda}}[A_\lambda]+\sum_{B\prec A} p_{B,A}(-1)^{\overline{B_\lambda}}[B_\lambda].
\endaligned
$$
By definition, $B\prec A$ implies $B_\lambda\prec A_\lambda$. Also, $\la=\co(A_\la)$. Thus, if $\mu=\ro(A_\la)$, then
$\sfc_A[\diag(\la)]=[\diag(\mu)]\sfc_A$.
By Corollary \ref{image of canonical basis}, we obtain
$$(-1)^{\overline{A_\lambda}}\sfc_A[\diag(\la)]=[A_\lambda]+\sum_{B_\lambda\prec_{\rc} A_\lambda}(-1)^{\overline{A_\lambda}+\overline{B_\lambda}} p_{B,A}[{B_\lambda}],$$
where $\preceq_{\rc}$ is the partial order relation on $M(m|n,r)$ defined in \eqref{prec rc}. Now, by Remark \ref{Xi relative to order rc} and the uniqueness of canonical basis, we must have $(-1)^{\overline{A_\lambda}}\sfc_A[\diag(\la)]=\Xi_{A_\la}$.
\end{proof}

We end this section with description of a PBW type basis for $\sS(m|n,r)$.

\begin{corollary}\label{PBW for S} Maintain the notation in \eqref{Ala}.
For $A\in M(m|n)^\pm$ and $\la\in\La(m|n,r)$, we have
\begin{equation*}
A^-(\bfl,r)[\diag(\la)]A^+(\bfl,r)
=\vep_{\la,\bsfh(A)}(-1)^{\overline{A_\la}}[A_\la]+(\text{lower terms w.r.t. $\preceq_\rc$}),
\end{equation*}
where $\vep_{\la,\bsfh(A)}=1$ if $\la\geq\bsfh(A)$ and 0 otherwise.
In particular, the set $$\{A^-(\bfl,r)[\diag(\la)]A^+(\bfl,r)\mid A\in M(m|n)^\pm,\,\la\in\La(m|n,r),\,\la\geq\bsfh(A)\}$$ forms a $\sZ$-basis for $\sS(m|n,r)$.
\end{corollary}
\begin{proof} For $A\in M(m|n)^\pm$, write $A=A^++A^-$ where
$A^+\in M(m|n)^+$ and $A^-\in M(m|n)^-$. Then $\mathfrak{h}(A)=\ro(A^+)+\co(A^-)$.
Since
$$A^-(\bfl,r)[\diag(\la)]A^+(\bfl,r)=A^-(\bfl,r)A^+(\bfl,r)[\diag(\la+\co(A^+)-\ro(A^+))]$$ and, by \cite[(8.1.1)]{DG} and the argument after \cite[(8.1.2)]{DG} (cf. \cite[Th.~7.1]{DG}, \cite[Th.~13.44]{DDPW}),
$$A^-(\bfl,r)A^+(\bfl,r)=A(\bfl,r)+(\text{lower terms w.r.t. $\preceq_\rc$}),$$
the first assertion follows the fact that $A(\bfl,r)[\diag(\la+\co(A^+)-\ro(A^+)]\neq 0$ implies $\co(A)+\mu=\la+\co(A^+)-\ro(A^+)$ for some $\mu\in\La(m|n,r-|A|)$ (so $\mu=\la-\bsfh(A)$). The last assertion is clear since, if $\la\geq\bsfh(A)$, then $A^-(\bfl,r)[\diag(\la)]A^+(\bfl,r)$ has the leading term $(-1)^{\overline{A_\la}}[A_\la]$.
\end{proof}

\section{Examples: Canonical bases for $U_\sZ^+(\mathfrak{gl}_{2|1})$ and $U_\sZ^+(\mathfrak{gl}_{2|2})$}
Recall from \S5 that the relation \eqref{monomial} was used to show the existence of canonical bases. In this section, we will see how this relation is used to compute the canonical basis element $\cC_A$.

Let $A\in M(m|n)^+$.  We use the order $\leq_1$ to write down the monomial $\fm_A^+$, i.e., the left hand side of \eqref{monomial}. Then apply the formula in Lemma \ref{pEA} to compute the right hand side of \eqref{monomial}:
$$
{\fm}^{+}_A=A({\bf0})+\sum_{B\in M(m|n)^+,B\prec A}g_{B,A}B({\bf0}) \quad(g_{B,A}\in\sZ).
$$
If all $g_{B,A}\in\sZ^-:=\up^{-1}\mathbb{Z}[\up^{-1}]$, then
by Theorem \ref{canonical basis}, $\cC_A=\fm_A^+$. Suppose now some $g_{B,A}\not\in\sZ^-$.
Partition the poset ideal
$$\sI_{\prec A}=\{B\in M(m|n)^+\mid B\prec A\}=\sI^1_{\prec A}\cup\sI^2_{\prec A}\cup\cdots\cup\sI^t_{\prec A},$$
into subsets $\sI^i_{\prec A}$ which consists of the maximal elements of $\sI_{\prec A}\setminus\bigcup_{j=1}^{i-1}\sI^j_{\prec A}$ for all $1\leq i\leq t$. In particular, $\sI^1_{\prec A}$ consists of all maximal elements of $\sI_{\prec A}$. Choose $B$ so that $g_{B,A}\not\in\sZ^-$ and $B\in\sI_{\prec A}^a$ with $a$ minimal. Thus, $g_{B',A}\in\sZ^-$ if $B\prec B'$ or $B'\in\sI_{\prec A}^i$ for some $i<a$.
Since every polynomial $g\in\sZ$ can be written as $g'+g''$ with $g'\in\sY:=\{h+\bar h\mid h\in\mathbb{Z}[\up]\}$ and $g''\in\sZ^-$, there exist $f_{B,A}\in \sZ$ and $g'_{B,A}\in\sY$ such that

\begin{equation*}\label{equation2}
{\fm}^+_{A}-\sum_{B\in \sI^a_{\prec A}}g'_{B,A} {\fm}^+_{B}
-\bigg(A({\bf0})+\sum_{B\in\sI^i_{\prec A},i>a}f_{B,A}B({\bf0})\bigg)\in\sum_{B\prec A}\sZ^-B(\bfl).
\end{equation*}
By a similar argument with $f_{B,A}$ and continuing if necessary, we can eventually find an $\fm\in\sum_{B\prec A}\sY\fm_B^+$ such that
$${\fm}^+_{A}+\fm
-A({\bf0})\in\sum_{B\prec A}\sZ^-B(\bfl).$$
Since $\overline{{\fm}^+_{A}+\fm}={\fm}^+_{A}+\fm$, we must have $\cC_A={\fm}^+_{A}+\fm$  by Theorem \ref{canonical basis}. We have proved the following.

\begin{lemma} \label{m=c} For $A\in M(m|n)^+$, there exists $\fm\in\sum_{B\prec A}\sY\fm_B^+$ such that the canonical basis element $\cC_A={\fm}^+_{A}+\fm$.
\end{lemma}

This algorithm has been used before (see, e.g.,  \cite{CX}). We now use the algorithm to compute some small rank examples.
\begin{example}
The canonical basis of $U_\sZ^+(\mathfrak{gl}_{2|1})$ has been found in \cite{Kh} and \cite[8.1]{CHW2}. We now follow the algorithm via our multiplication formulas to show that it consists of
\begin{equation}\label{canonical A2}
\begin{aligned}
&\sfE_{1}^{(a)}=(aE_{1,2})({\bf0}),\\
&\sfE_{2}\sfE^{(a)}_{1}=(E_{2,3}+aE_{1,2})({\bf0}),\\
&\sfE_{1}\sfE_{2}\sfE^{(a)}_{1}-[a]\sfE_{2}\sfE^{(a+1)}_{1}=(aE_{1,2}+E_{1,3})({\bf0})+\up^{-(a+1)}((a+1)E_{1,2}+E_{2,3})({\bf0}),\\
&\sfE_{2}\sfE_{1}\sfE_{2}\sfE^{(a)}_{1}=(aE_{1,2}+E_{1,3}+E_{2,3})({\bf0}).
\end{aligned}
\end{equation}
We observe the following.
From \cite[Example 3.4]{Lu3}, $\bfU(\mathfrak{gl}_{3})$ has the canonical basis consisting of tight monomials
$$\{\sfE_{1}^{(b)}\sfE_{2}^{(b+c)}\sfE_{1}^{(a)}\mid c\geq a\}\cup\{\sfE_{2}^{(c)}\sfE_{1}^{(a+b)}\sfE_{2}^{(b)}\mid c<a\}.$$
If we regard $\sfE_2$ as an odd generator and only consider
the power $\sfE_{2}^{(a)}$ with $a=0,1$,  we  obtain the following elements from the classical canonical basis above:
$$\sfE_{1}^{(a)},\quad \sfE_{1}^{(a+1)}\sfE_{2},\quad\sfE_{2}\sfE_{1}^{(a)},\quad\sfE_{2}\sfE_{1}^{(a+1)}\sfE_{2}.$$
Then we claim that they coincide with the canonical basis above. Indeed, by
the multiplication formulas given in Lemma \ref{pEA}, we have
\begin{equation*}
\begin{aligned}
& \sfE_{1}^{(a+1)}\sfE_{2}=(aE_{1,2}+E_{1,3})({\bf0})+\up^{-(a+1)}((a+1)E_{1,2}+E_{2,3})({\bf0}).\\
&\sfE_{2}\sfE_{1}^{(a+1)}\sfE_{2}=(aE_{1,2}+E_{1,3}+E_{2,3})({\bf0}),
\end{aligned}
\end{equation*}
\end{example}
\noindent
which are the third and fourth elements in  \eqref{canonical A2}. Hence, the canonical basis for $U_\sZ^+(\mathfrak{gl}_{2|1})$ consists of tight monomials.

We now compute the  canonical basis of $U_\sZ^+(\mathfrak{gl}_{2|2})$.
We will use the following abbreviation for a $4\times4$ matrix in $M(2|2)^+$:
$$A=
\left[
\begin{array}{ccc}
a&b&d\\
&c&e\\
&&f
\end{array}
\right]:=\left(
\begin{array}{cccc}
0&a&b&d\\
0&0&c&e\\
0&0&0&f\\
0&0&0&0
\end{array}
\right)\in M(2|2)^+
$$
where $a,f\in\mathbb{Z}_{\geq0},b,c,d,e\in \{0,1\}$.

\begin{example} The canonical basis of $U_\sZ^+(\mathfrak{gl}_{2|2})$ is listed in the following 18 cases. Each case is displayed in the form: $A$, $\fm_A^++\fm=\cC_A$ as in Lemma \ref{m=c} where $\fm$ is a $\sY$-linear combinations of monomial basis.
\begin{itemize}
\item[(0)]$A=\left[
\begin{smallmatrix}
a&0&0\\
&0&0\\
&&f
\end{smallmatrix}
\right],
$
$\sfE_{3}^{(f)}\sfE_{1}^{(a)}=A({\bf0}).$ This is the only even case.

\item[(1)] $A=\left[
\begin{smallmatrix}
a&0&0\\
&1&0\\
&&f
\end{smallmatrix}
\right],
$
$\sfE_{3}^{(f)}\sfE_{2}\sfE_{1}^{(a)}=A({\bf0}).$

\item[(2)]  $A=\left[
\begin{smallmatrix}
a&1&0\\
&0&0\\
&&f
\end{smallmatrix}
\right],
$
$\sfE_{3}^{(f)}\sfE_{1}\sfE_{2}\sfE_{1}^{(a)}-[a]\sfE_{3}^{(f)}\sfE_{2}\sfE_{1}^{(a+1)}
=A({\bf0})+\up^{-a-1}\left[
\begin{smallmatrix}
a+1&0&0\\
&1&0\\
&&f
\end{smallmatrix}
\right]({\bf0}).$

\item[(3)] $A=\left[
\begin{smallmatrix}
a&0&0\\
&0&1\\
&&f
\end{smallmatrix}
\right],
$
$\sfE_{3}^{(f)}\sfE_{2}\sfE_{3}\sfE_{1}^{(a)}-[f+2]\sfE_{3}^{(f+1)}\sfE_{2}\sfE_{1}^{(a)}
=A({\bf0})-\up^{-f-1}\left[
\begin{smallmatrix}
a&0&0\\
&1&0\\
&&f+1
\end{smallmatrix}
\right]({\bf0}).$

\item[(4a)] $A=\left[
\begin{smallmatrix}
a&0&1\\
&0&0\\
&&f
\end{smallmatrix}
\right]
$ with $a\leq f$,
$$\aligned
&\quad\,\sfE_{3}^{(f)}\sfE_{1}\sfE_{2}\sfE_{3}\sfE_{1}^{(a)}-[a]\sfE_{3}^{(f)}\sfE_{2}\sfE_{3}\sfE_{1}^{(a+1)}
-[f+2]\sfE_{3}^{(f)}\sfE_{1}\sfE_{2}\sfE_{1}^{(a+1)}\\
&\quad\,+(2[a][f+2]+[f-a+1]-[a+1][f+1])\sfE_{3}^{(f+1)}\sfE_{2}\sfE_{1}^{(a+1)}\\
&=A({\bf0})
+\up^{-a-1}\left[
\begin{smallmatrix}
a+1&0&0\\
&0&1\\
&&f
\end{smallmatrix}
\right]({\bf0})
-\up^{-f-1}
\left[
\begin{smallmatrix}
a+1&1&0\\
&0&0\\
&&f+1
\end{smallmatrix}\right]({\bf0})
-\up^{-f-a-2}
\left[
\begin{smallmatrix}
a+1&0&0\\
&1&0\\
&&f+1
\end{smallmatrix}\right]({\bf0})
\endaligned
$$

\item[(4b)] $A=\left[
\begin{smallmatrix}
a&0&1\\
&0&0\\
&&f
\end{smallmatrix}
\right]
$ with $f=a-1$,
$$\aligned
&\quad\,\sfE_{3}^{(f)}\sfE_{1}\sfE_{2}\sfE_{3}\sfE_{1}^{(a)}-[a]\sfE_{3}^{(f)}\sfE_{2}\sfE_{3}\sfE_{1}^{(a+1)}
-[f+2]\sfE_{3}^{(f)}\sfE_{1}\sfE_{2}\sfE_{1}^{(a+1)}\\
&\quad\,+(2[a][f+2]-[a+1][f+1])\sfE_{3}^{(f+1)}\sfE_{2}\sfE_{1}^{(a+1)}\\
&=A({\bf0})
+\up^{-a-1}
\left[
\begin{smallmatrix}
a+1&0&0\\
&0&1\\
&&f
\end{smallmatrix}
\right]({\bf0})
-\up^{-f-1}
\left[
\begin{smallmatrix}
a+1&1&0\\
&0&0\\
&&f+1
\end{smallmatrix}\right]({\bf0})
+(*)
\left[
\begin{smallmatrix}
a+1&0&0\\
&1&0\\
&&f+1
\end{smallmatrix}\right]({\bf0})
\endaligned
$$
where $*=\up^{-f-1}[a]-\up^{-a-1}[f+2]$.

\item[(4c)] $A=\left[
\begin{smallmatrix}
a&0&1\\
&0&0\\
&&f
\end{smallmatrix}
\right]
$ with
 $f\leq a-2$,
$$\aligned
&\quad\,\sfE_{3}^{(f)}\sfE_{1}\sfE_{2}\sfE_{3}\sfE_{1}^{(a)}-[a]\sfE_{3}^{(f)}\sfE_{2}\sfE_{3}\sfE_{1}^{(a+1)}
-[f+2]\sfE_{3}^{(f)}\sfE_{1}\sfE_{2}\sfE_{1}^{(a+1)}\\
&\quad\,+(2[a][f+2]-[a-f-1]-[a+1][f+1])\sfE_{3}^{(f+1)}\sfE_{2}\sfE_{1}^{(a+1)}\\
&=A({\bf0})
+\up^{-a-1}\left[
\begin{smallmatrix}
a+1&0&0\\
&0&1\\
&&f
\end{smallmatrix}
\right]({\bf0})
-\up^{-f-1}
\left[
\begin{smallmatrix}
a+1&1&0\\
&0&0\\
&&f+1
\end{smallmatrix}\right]({\bf0})
-\up^{-f-a-2}
\left[
\begin{smallmatrix}
a+1&0&0\\
&1&0\\
&&f+1
\end{smallmatrix}\right]({\bf0}).
\endaligned
$$

\item[(5)] $A=\left[
\begin{smallmatrix}
a&1&0\\
&1&0\\
&&f
\end{smallmatrix}
\right],
$
$\sfE_{3}^{(f)}\sfE_{2}\sfE_{1}\sfE_{2}\sfE_{1}^{(a)}=A({\bf0}).$

\item[(6)] $A=\left[
\begin{smallmatrix}
a&0&0\\
&1&1\\
&&f
\end{smallmatrix}
\right],
$ $\sfE_{3}^{(f)}\sfE_{2}\sfE_{3}\sfE_{2}\sfE_{1}^{(a)}=A({\bf0}).$
\vspace{-2ex}

\item[(7)] $A=\left[
\begin{smallmatrix}
a&1&0\\
&0&1\\
&&f
\end{smallmatrix}
\right],
$
$\aligned
&\\
&\sfE_{3}^{(f)}\sfE_{2}\sfE_{3}\sfE_{1}\sfE_{2}\sfE_{1}^{(a)}-[a]\sfE_{3}^{(f)}\sfE_{2}\sfE_{3}\sfE_{2}\sfE_{1}^{(a+1)}
-[f+2]\sfE_{3}^{(f+1)}\sfE_{2}\sfE_{1}\sfE_{2}\sfE_{1}^{(a)}\\
&=A({\bf0})+\up^{-a-1}
\left[
\begin{smallmatrix}
a+1&0&0\\
&1&1\\
&&f
\end{smallmatrix}
\right]({\bf0})
-\up^{-f-1}
\left[
\begin{smallmatrix}
a&1&0\\
&1&0\\
&&f+1
\end{smallmatrix}
\right]({\bf0})
\endaligned
$
\vspace{.2cm}

\item[(8a)] $A=\left[
\begin{smallmatrix}
a&0&1\\
&1&0\\
&&f
\end{smallmatrix}
\right] (a=0),
$
$\sfE_{3}^{(f)}\sfE_{1}\sfE_{2}\sfE_{3}\sfE_{2}=A({\bf0})
+\up^{-1}
\left[
\begin{smallmatrix}
0&1&0\\
&0&1\\
&&f+1
\end{smallmatrix}
\right]({\bf0})
+\up^{-2}
\left[
\begin{smallmatrix}
1&0&0\\
&1&1\\
&&f
\end{smallmatrix}
\right]({\bf0}).
$

\item[(8b)] $A=\left[
\begin{smallmatrix}
a&0&1\\
&1&0\\
&&f
\end{smallmatrix}
\right] (a>0),
$
$\aligned
&\\
&\sfE_{3}^{(f)}\sfE_{1}\sfE_{2}\sfE_{3}\sfE_{2}\sfE_{1}^{(a)}
-[a-1]\sfE_{3}^{(f)}\sfE_{2}\sfE_{3}\sfE_{2}\sfE_{1}^{(a+1)}\\
&=A({\bf0})
+\up^{-1}
\left[
\begin{smallmatrix}
a&1&0\\
&0&1\\
&&f+1
\end{smallmatrix}
\right]({\bf0})
+(\up^{-a}+\up^{-a-2})
\left[
\begin{smallmatrix}
a+1&0&0\\
&1&1\\
&&f
\end{smallmatrix}
\right]({\bf0}).
\endaligned
$

\item[(9)] $A=\left[
\begin{smallmatrix}
a&1&1\\
&0&0\\
&&f
\end{smallmatrix}
\right],
$
$$\aligned
&\quad\,\sfE_{3}^{(f)}\sfE_{1}\sfE_{2}\sfE_{3}\sfE_{1}\sfE_{2}\sfE_{1}^{(a)}
-[a]\sfE_{3}^{(f)}\sfE_{1}\sfE_{2}\sfE_{3}\sfE_{2}\sfE_{1}^{(a+1)}\\
&\quad\,-[a+1]\sfE_{3}^{(f)}\sfE_{2}\sfE_{3}\sfE_{1}\sfE_{2}\sfE_{1}^{(a+1)}
+[a+1]^2\sfE_{3}^{(f)}\sfE_{2}\sfE_{3}\sfE_{2}\sfE_{1}^{(a+2)}\\
&=A({\bf0})
+\up^{-a-1}
\left[
\begin{smallmatrix}
a+1&0&1\\
&1&0\\
&&f
\end{smallmatrix}
\right]({\bf0})
+\up^{-a-2}
\left[
\begin{smallmatrix}
a+1&1&0\\
&0&1\\
&&f
\end{smallmatrix}
\right]({\bf0})
+\up^{-2a-4}
\left[
\begin{smallmatrix}
a+2&0&0\\
&1&1\\
&&f
\end{smallmatrix}
\right]({\bf0}).
\endaligned$$

\item[(10)]
$A=\left[
\begin{smallmatrix}
a&1&0\\
&1&1\\
&&f
\end{smallmatrix}
\right],
$ $\sfE_{3}^{(f)}\sfE_{2}\sfE_{3}\sfE_{2}\sfE_{1}\sfE_{2}\sfE_{1}^{(a)}=A({\bf0}).$
%\vspace{-2ex}

\item[(11)] $A=\left[
\begin{smallmatrix}
a&1&1\\
&1&0\\
&&f
\end{smallmatrix}
\right],
$
$\aligned
&\\
&\sfE_{3}^{(f)}\sfE_{1}\sfE_{2}\sfE_{3}\sfE_{2}\sfE_{1}\sfE_{2}\sfE_{1}^{(a)}
-[a]\sfE_{3}^{(f)}\sfE_{2}\sfE_{3}\sfE_{2}\sfE_{1}\sfE_{2}\sfE_{1}^{(a+1)}\\
&=A({\bf0})+\up^{-a-1}\left[
\begin{smallmatrix}
a&1&0\\
&1&1\\
&&f
\end{smallmatrix}
\right]({\bf0})
\endaligned
$
\vspace{-2ex}

\item[(12)]
$A=\left[
\begin{smallmatrix}
a&0&1\\
&1&1\\
&&f
\end{smallmatrix}
\right],
$
$\aligned
&\\
&\sfE_{3}^{(f)}\sfE_{2}\sfE_{3}\sfE_{1}\sfE_{2}\sfE_{3}\sfE_{2}\sfE_{1}^{(a)}
+[f+2]\sfE_{3}^{(f+1)}\sfE_{2}\sfE_{3}\sfE_{2}\sfE_{1}\sfE_{2}\sfE_{1}^{(a)}\\
&=A({\bf0})+\up^{-f-1}
\left[
\begin{smallmatrix}
a&1&0\\
&1&1\\
&&f+1
\end{smallmatrix}
\right]({\bf0}).
\endaligned
$

\item[(13a)] $A=\left[
\begin{smallmatrix}
a&1&1\\
&0&1\\
&&f
\end{smallmatrix}
\right]
$ with $a\leq f$,
$$\aligned
&\quad\;\sfE_{3}^{(f)}\sfE_{2}\sfE_{3}\sfE_{1}\sfE_{2}\sfE_{3}\sfE_{1}\sfE_{2}\sfE_{1}^{(a)}+[f+2]\sfE_{3}^{(f+1)}\sfE_{1}\sfE_{2}\sfE_{3}\sfE_{2}\sfE_{1}\sfE_{2}\sfE_{1}^{(a)}\\
&\quad\;-[a]\sfE_{3}^{(f)}\sfE_{2}\sfE_{3}\sfE_{1}\sfE_{2}\sfE_{3}\sfE_{2}\sfE_{1}^{(a+1)}
-(2[a][f+2]+[f-a+1])\sfE_{3}^{(f+1)}\sfE_{2}\sfE_{3}\sfE_{2}\sfE_{1}\sfE_{2}\sfE_{1}^{(a+1)}\\
&=A({\bf0})
+\up^{-f-1}
\left[
\begin{smallmatrix}
a&1&1\\
&1&0\\
&&f+1
\end{smallmatrix}
\right]({\bf0})
+\up^{-a-1}
\left[
\begin{smallmatrix}
a+1&0&1\\
&1&1\\
&&f
\end{smallmatrix}
\right]({\bf0})
+\up^{-f-a-2}
\left[
\begin{smallmatrix}
a+1&1&0\\
&1&1\\
&&f
\end{smallmatrix}
\right]({\bf0})
\endaligned$$

\item[(13b)] $A=\left[
\begin{smallmatrix}
a&1&1\\
&0&1\\
&&f
\end{smallmatrix}
\right]
$ with $f=a-1$,
$$\aligned
&\quad\;\sfE_{3}^{(f)}\sfE_{2}\sfE_{3}\sfE_{1}\sfE_{2}\sfE_{3}\sfE_{1}\sfE_{2}\sfE_{1}^{(a)}
+[f+2]\sfE_{3}^{(f+1)}\sfE_{1}\sfE_{2}\sfE_{3}\sfE_{2}\sfE_{1}\sfE_{2}\sfE_{1}^{(a)}\\
&\quad\;-[a]\sfE_{3}^{(f)}\sfE_{2}\sfE_{3}\sfE_{1}\sfE_{2}\sfE_{3}\sfE_{2}\sfE_{1}^{(a+1)}
-2[a][f+2]\sfE_{3}^{(f+1)}\sfE_{2}\sfE_{3}\sfE_{2}\sfE_{1}\sfE_{2}\sfE_{1}^{(a+1)}\\
&=A({\bf0})
+\up^{-f-1}
\left[
\begin{smallmatrix}
a&1&1\\
&1&0\\
&&f+1
\end{smallmatrix}
\right]({\bf0})
+\up^{-a-1}
\left[
\begin{smallmatrix}
a+1&0&1\\
&1&1\\
&&f
\end{smallmatrix}
\right]({\bf0})
+(**)
\left[
\begin{smallmatrix}
a+1&1&0\\
&1&1\\
&&f
\end{smallmatrix}
\right]({\bf0})
\endaligned$$
where $(**)=\up^{-a-1}[f+2]-\up^{-f-1}[a]$

\item[(13c)] $A=\left[
\begin{smallmatrix}
a&1&1\\
&0&1\\
&&f
\end{smallmatrix}
\right]
$ with $f\leq a-2$,
$$\aligned
&\quad\;\sfE_{3}^{(f)}\sfE_{2}\sfE_{3}\sfE_{1}\sfE_{2}\sfE_{3}\sfE_{1}\sfE_{2}\sfE_{1}^{(a)}+[f+2]\sfE_{3}^{(f+1)}\sfE_{1}\sfE_{2}\sfE_{3}\sfE_{2}\sfE_{1}\sfE_{2}\sfE_{1}^{(a)}\\
&\quad\;-[a]\sfE_{3}^{(f)}\sfE_{2}\sfE_{3}\sfE_{1}\sfE_{2}\sfE_{3}\sfE_{2}\sfE_{1}^{(a+1)}-(2[a][f+2]-[a-\!f\!-1])\sfE_{3}^{(f+1)}\sfE_{2}\sfE_{3}\sfE_{2}\sfE_{1}\sfE_{2}\sfE_{1}^{(a+1)}\\
&=A({\bf0})
+\up^{-f-1}
\left[
\begin{smallmatrix}
a&1&1\\
&1&0\\
&&f+1
\end{smallmatrix}
\right]({\bf0})
+\up^{-a-1}
\left[
\begin{smallmatrix}
a+1&0&1\\
&1&1\\
&&f
\end{smallmatrix}
\right]({\bf0})
+\up^{-f-a-2}
\left[
\begin{smallmatrix}
a+1&1&0\\
&1&1\\
&&f
\end{smallmatrix}
\right]({\bf0})
\endaligned$$

\item[(14)] $A=\left[
\begin{smallmatrix}
a&1&1\\
&1&1\\
&&f
\end{smallmatrix}
\right],
$ $\sfE_{3}^{(f)}\sfE_{2}\sfE_{3}\sfE_{1}\sfE_{2}\sfE_{3}\sfE_{2}\sfE_{1}\sfE_{2}\sfE_{1}^{(a)}=A({\bf0}).$
\end{itemize}
\end{example}
\begin{proof}
We just give a proof for (9). The other cases can be proved in a similar way.

For $A=\left[
\begin{smallmatrix}
a&1&1\\
&0&0\\
&&f
\end{smallmatrix}
\right]=aE_{1,2}+E_{1,3}+E_{1,4}+fE_{3,4}$,  by definition,
$${\fm}^{+}_{A}=\sfE^{(f)}_{3,4}\cdot\sfE_{1,2}\sfE_{2,3}\sfE_{3,4}\cdot\sfE_{1,2}\sfE_{2,3}\cdot\sfE_{1,2}^{(a)}.$$
Repeatedly applying the multiplication formula in Lemma \ref{integral generators} yields
\begin{equation}
\begin{aligned}
{\fm}^{+}_{A}=&\sfE^{(f)}_{3,4}\big((aE_{1,2}+E_{1,3}+E_{1,4})({\bf0})\\
&+\up^{a-1}\overline{[\![a+1]\!]}((a+1)E_{1,2}+E_{2,3}+E_{1,4})({\bf0})\\
&+(\up^{a-2}\overline{[\![a+1]\!]}+\up^{a}\overline{[\![a+1]\!]})((a+1)E_{1,2}+E_{1,3}+E_{2,4})({\bf0})\\
&+\up^{a-1}\overline{[\![a+1]\!]}\up^{a-1}\overline{[\![a+2]\!]}((a+2)E_{1,2}+E_{2,3}+E_{2,4})({\bf0})\\
&+\up^{a+1}\overline{[\![a+1]\!]}[f+1]((a+1)E_{1,2}+E_{2,3}+E_{1,3}+E_{3,4})({\bf0})\big)
\end{aligned}
\end{equation}

For $C=(c_{i,j})\in M(2|2)^+$, by the multiplication formula in Lemma \ref{pEA},
 $$\sfE^{(f)}_{3,4}C({\bf0})=\up_3^{fc_{3,4}}\overline{\left[\!\!\left[f+c_{3,4}\atop f\right]\!\!\right]}_{\up_3}(C+fE_{3,4})({\bf0}).$$
Then, noting $\up_3^{f}\overline{\left[\!\!\left[f+1\atop f\right]\!\!\right]}_{\up_3}=[f+1]_{\up_3}=[f+1],$
$$\sfE^{(f)}_{3,4}C({\bf0})=\begin{cases}(C+fE_{3,4})({\bf0}),&\text{ if }c_{3,4}=0;\\
[f+1](C+fE_{3,4})({\bf0}),&\text{ if }c_{3,4}=1.\end{cases}$$
Thus,
\begin{equation}
\begin{aligned}
{\fm}^{+}_{A}
=&A({\bf0})
+\up^{a-1}\overline{[\![a+1]\!]}((a+1)E_{1,2}+E_{2,3}+E_{1,4}+fE_{3,4})({\bf0})\\
&+(\up^{a-2}\overline{[\![a+1]\!]}+\up^{a}\overline{[\![a+1]\!]})((a+1)E_{1,2}+E_{1,3}+E_{2,4}+fE_{3,4})({\bf0})\\
&+\up^{a-1}\overline{[\![a+1]\!]}\up^{a-1}\overline{[\![a+2]\!]}((a+2)E_{1,2}+E_{2,3}+E_{2,4}+fE_{3,4})({\bf0})\\
&+\up^{a+1}\overline{[\![a+1]\!]}[f+1]((a+1)E_{1,2}+E_{2,3}+E_{1,3}+(f+1)E_{3,4})({\bf0})
\end{aligned}
\end{equation}

Observing the summands above, the maximal matrix $B$ such that $B\prec A$ is
 $B_1=(a+1)E_{1,2}+E_{2,3}+E_{1,4}+fE_{3,4}$, and the coefficient of $B_1({\bf0})$ is
$$\up^{a-1}\overline{[\![a+1]\!]}=[a]+\up^{-a-1}.$$
So we compute ${\fm}^{+}_{A}-[a]{\fm}^{+}_{B_1}$. By the multiplication formulas, we have
\begin{equation*}
\begin{aligned}
{\fm}^{+}_{B_1}=&\sfE^{(f)}_{3,4}\sfE_{1,2}\sfE_{2,3}\sfE_{3,4}\sfE_{2,3}\sfE_{1,2}^{(a+1)}\\
=&B_1({\bf0})+\up^{-1}((a+1)E_{1,2}+E_{1,3}+E_{2,4}+fE_{3,4})({\bf0})\\
&+\up^{a-1}\overline{[\![a+2]\!]}((a+2)E_{1,2}+E_{2,3}+E_{2,4}+fE_{3,4})({\bf0}).
\end{aligned}
\end{equation*}
Then
\begin{equation}\label{MAMB1}
\aligned
{\fm}^{+}_{A}&-[a]{\fm}^{+}_{B_1}=A({\bf0})+\up^{-a-1}B_1({\bf0})\\
&+(\up^{a-2}\overline{[\![a+1]\!]}+\up^{a}\overline{[\![a+1]\!]}-\up^{-1}[a])((a+1)E_{1,2}+E_{1,3}+E_{2,4}+fE_{3,4})({\bf0})\\
&+(\up^{a-1}\overline{[\![a+1]\!]}\up^{a-1}\overline{[\![a+2]\!]}-\up^{a-1}\overline{[\![a+2]\!]}[a])((a+2)E_{1,2}+E_{2,3}+E_{2,4}+fE_{3,4})({\bf0})\\
&+\up^{a+1}\overline{[\![a+1]\!]}[f+1]((a+1)E_{1,2}+E_{2,3}+E_{1,3}+(f+1)E_{3,4})({\bf0})
\endaligned
\end{equation}
Now the maximal matrices $B$ in \eqref{MAMB1} such that $B\prec B_1$ is $$B_2=(a+1)E_{1,2}+E_{1,3}+E_{2,4}+fE_{3,4}.$$
Since the coefficient of $B_2({\bf0})$ in \eqref{MAMB1} is
$$\aligned
\up^{a-2}\overline{[\![a+1]\!]}&+\up^{a}\overline{[\![a+1]\!]}-\up^{-1}[a]\\
&=[a-1]+\up^{-a}+\up^{-a-2}+[a+1]-[a-1]-\up^{-a}=\up^{-a-2}+[a+1],
\endaligned$$
We now compute ${\fm}^{+}_{A}-[a]{\fm}^{+}_{B_1}-[a+1]{\fm}^{+}_{B_2}$. Since
\begin{equation*}
\begin{aligned}
{\fm}^{+}_{B_2}=&\sfE^{(f)}_{3,4}\sfE_{2,3}\sfE_{3,4}\sfE_{1,2}\sfE_{2,3}\sfE_{1,2}^{(a+1)}\\
=&B_2({\bf0})+\up^a\overline{[\![a+2]\!]}((a+2)E_{1,2}+E_{2,3}+E_{2,4}+fE_{3,4})({\bf0})\\
&+\up[f+1]((a+1)E_{1,2}+E_{2,3}+E_{1,3}+(f+1)E_{3,4})({\bf0}),
\end{aligned}
\end{equation*}
it yields
\begin{equation*}
\aligned
{\fm}^{+}_{A}&-[a]{\fm}^{+}_{B_1}-[a+1]{\fm}^{+}_{B_2}=A({\bf0})+\up^{-a-1}B_1({\bf0})+\up^{-a-2}B_2({\bf0})\\
&+(\up^{a-1}\overline{[\![a+1]\!]}\up^{a-1}\overline{[\![a+2]\!]}-\up^{a-1}\overline{[\![a+2]\!]}[a]-\up^a[a+1]\overline{[\![a+2]\!]})\\
&((a+2)E_{1,2}+E_{2,3}+E_{2,4}+fE_{3,4})({\bf0})\\
&+(\up^{a+1}\overline{[\![a+1]\!]}[f+1]-\up[a+1][f+1])((a+1)E_{1,2}+E_{2,3}+E_{1,3}+(f+1)E_{3,4})({\bf0})
\endaligned
\end{equation*}
But the coefficient of $((a+1)E_{1,2}+E_{2,3}+E_{1,3}+(f+1)E_{3,4})({\bf0})$ is
$$\up^{a+1}\overline{[\![a+1]\!]}[f+1]-[a+1]\up[f+1]=\up[a+1][f+1]-\up[a+1][f+1]=0,$$
so
\begin{equation}
\aligned
{\fm}^{+}_{A}&-[a]{\fm}^{+}_{B_1}-[a+1]{\fm}^{+}_{B_2}=A({\bf0})+\up^{-a-1}B_1({\bf0})+\up^{-a-2}B_2({\bf0})\\
&+(\up^{a-1}\overline{[\![a+1]\!]}\up^{a-1}\overline{[\![a+2]\!]}-\up^{a-1}\overline{[\![a+2]\!]}[a]-\up^a[a+1]\overline{[\![a+2]\!]})\\
&\quad\;((a+2)E_{1,2}+E_{2,3}+E_{2,4}+fE_{3,4})({\bf0})\\
\endaligned
\end{equation}
Let $B_3=(a+2)E_{1,2}+E_{2,3}+E_{2,4}+fE_{3,4}$ and rewrite coefficient of $B_3({\bf0})$ as
$$\aligned
\up^{a-1}\overline{[\![a+1]\!]}&\up^{a-1}\overline{[\![a+2]\!]}-\up^{a-1}\overline{[\![a+2]\!]}[a]-\up^a[a+1]\overline{[\![a+2]\!]}\\
&=\up^{-a-1}[a]-\up^{-a-2}[a+1]+\up^{-2a-2}+\up^{-2a-4}-[a+1]^2\\
&=\up^{-2a-4}-[a+1]^2.
\endaligned
$$
Finally, we compute ${\fm}^{+}_{A}-[a]{\fm}^{+}_{B_1}-[a+1]{\fm}^{+}_{B_2}+[a+1]^2 {\fm}^{+}_{B_3}.$ Since
$$
{\fm}^{+}_{B_3}=\sfE^{(f)}_{3,4}\sfE_{2,3}\sfE_{3,4}\sfE_{2,3}\sfE_{1,2}^{(a+2)}
=B_3({\bf0}),
$$
it follows that
$$
\aligned
{\fm}^{+}_{A}-[a]{\fm}^{+}_{B_1}&-[a+1]{\fm}^{+}_{B_2}+[a+1]^2 {\fm}^{+}_{B_3}\\
&=A({\bf0})+\up^{-a-1}B_1({\bf0})+\up^{-a-2}B_2({\bf0})+\up^{-2a-4}B_3({\bf0})\\
\endaligned
$$
is the required canonical basis element $\cC_A$.
\end{proof}

\section{Simple polynomial representations of $\bfU(\mathfrak{gl}_{m|n})$}

For a finite dimensional $\bfU(\mathfrak{gl}_{m|n})$-module $M$ and $\la\in\mathbb Z^{m+n}$, let
$$M_\la=\bigg\{x\in M\mid K_ix=\up_i^{\la_i}x, 1\leq i\leq m+n\bigg\}.$$
If $M_\la\neq 0$, then $M_\la$ is called the weight space of $M$ of weight $\la$.
Call $M$ an {\it integral weight module} (of type {\bf 1}) if $M=\bigoplus_\la M_\la$ and denote by wt$(M)$ the set of all weights of $M$. A weight module $M$ is called a {\it polynomial representation} of $\bfU(\mathfrak{gl}_{m|n})$ if
wt$(M)\subset \mathbb N^{m+n}$. Clearly, a tensor power of a polynomial representation is polynomial. In particular,  the tensor power $V^{\otimes r}$ of the natural representation $V$ of $U(\mathfrak{gl}_{m|n})$ is a polynomial representation.

Let $\bfU=\bfU(\mathfrak{gl}_{m|n})$ and $\bfU_{\bar0}=\bfU(\mathfrak{gl}_{m}\oplus\mathfrak{gl}_n)$ and $\bfU_{\bar 1}^\pm=\bfU(\mathfrak{gl}^\pm_{m|n,\bar1})$.
For $\lambda\in \La^+(m|n)$, let $L^{\bar0}(\lambda)$ be the (finite dimensional) irreducible module of $ \bfU_{\bar{0}}$ with the highest
weight $\lambda$. Then $L^{\bar0}(\lambda)$ becomes  a
module of the parabolic superalgebra $\bfU_{\bar{0}} \bfU^+_{\bar{1}}$ via the trivial action of $E_{a,b}$ on $L^{\bar0}(\lambda)$ for all $1\leq a\leq m<b\leq m+n$.
Define  the {\it Kac--module} (see \cite{Z})
$$K(\lambda)=\mathrm{Ind}^{\bfU}_{\bfU_{\bar{0}} \bfU^+_{\bar{1}}}L^{\bar0}(\lambda)=\bfU\otimes_{\bfU_{\bar{0}} \bfU^+_{\bar{1}}}L^{\bar0}(\la).$$
Since $ \bfU$ is a free $ \bfU_{\bar{0}} \bfU^+_{\bar{1}}$ module,
as vector spaces, we have
$$K(\lambda)\cong  \bfU^-_{\bar{1}}\otimes L^{\bar0}(\lambda).$$
Note that, for all $\mu\in\text{wt}(K(\la))$, $|\la|=|\mu|$ and $\mu\unlhd\la$ (meaning $\la-\mu$ is a sum of positive roots).\footnote{Since
$\la-\mu=(\la_1-\mu_1)\bse_1+\cdots+(\la_n-\mu_n)\bse_n=(\tilde\la_1-\tilde\mu_1)(\bse_1-\bse_2)+\cdots+(\tilde\la_{n-1}-\tilde\mu_{n-1})(\bse_{n-1}-\bse_n)$ ($\tilde a_j=\sum_{i=1}^ja_i$), this order is the usual dominance order $\unrhd$ if $\la,\mu$ are regarded as compositions.} Thus, we say that $K(\la)$ is a representation of $\bfU$ at {\it level} $|\la|$. Moreover, every $K(\la)$ has a unique maximal submodule and hence defines an simple module $L(\la)$. In fact, the set $\{L(\la)\mid\la\in\La^+(m|n)\}$ forms a complete set of finite dimensional simple $\bfU$-modules.

Since every  irreducible finite dimensional module $L(\lambda)$ of $\bfU$ is a quotient module of a Kac module $K(\lambda)$, $L(\la)$ is a representation at the same level as $K(\lambda)$.
\begin{lemma}\label{inflate}
The irreducible polynomial representations of $\bfU(\mathfrak{gl}_{m|n})$ at level $r\geq0$ are all inflated via $\eta_r$ from the irreducible representations of $\bsS(m|n,r)$.
\end{lemma}
\begin{proof} Clearly, if $M$ is an $\bsS(m|n,r)$-module, then $M=\bigoplus_{\la\in\La(m|n,r)}M_\la$ as a $\bfU(\mathfrak{gl}_{m|n})$-module, where $M_\la=\xi_\la M$ with $\xi_\la=[\diag(\la)]$. This is seen easily since $\eta_r(K_i)=\sum_{\la}\up_i^{\la_i}\xi_\la$. Hence, every inflated module is a module at level $r$.

 Assume now $M$ is an irreducible polynomial representation of $\bfU(\mathfrak{gl}_{m|n})$ at level $r$.
For any $0\neq x\in M_\mu$,
$$\aligned
K_1K_2\cdots K_mK^{-1}_{m+1}\cdots K^{-1}_{m+n}\cdot x&=\up^{\sum_{i=1}^{m+n}\mu_i}x= \up^rx\\
(K_i-1)(K_i-\up_i)\cdots (K_i-\up_i^r)\cdot x&=\prod_{j=0}^r(\up_i^{\mu_j}-\up_i^j)x=0x=0.\endaligned
$$
 By the presentation for $\bsS(m|n,r)$ given in \cite{TK}, we see that $M$ is in fact an inflation of a simple $\bsS(m|n,r)$-module.
\end{proof}

By this lemma, the study of simple polynomial representations of $\bfU(\mathfrak{gl}_{m|n})$ is reduced to that of simple $\bsS(m|n,r)$-modules for all $r\geq0$. Simple $\bsS(m|n,r)$-modules have been classified and constructed in \cite{DR} via a certain cellular basis.
We now use the cellular bases adjusted with a sign as in defining the canonical basis $\{\Xi_A\}_A$ to see how the canonical bases for $U_\sZ^-$ and $\sS(m|n,r)$ induce related bases for these modules.

For $A=\jmath(\la,d,\mu)$ as in \eqref{jmath}, define compositions $\al,\beta$ by \eqref{Dcirc},
\begin{equation}\label{beta}
\fS_{\al|\beta}:=\fS_{\la d\cap\mu}\cong (\fS_{\la^{(0)}}^d\cap\fS_{\mu^{(0)}})\times (\fS_{\la^{(1)}}^d\cap\fS_{\mu^{(1)}}).
\end{equation}

Let $\xi_A'=\up^{l(w_{0,\beta})}P_{\fS_\beta}(\up^{-2})\xi_A$ (cf. \cite[(6.3.1)]{DR}). Using this basis and the bar involution defined in \eqref{bar on S} (cf. \cite[Th.~6.3]{DR}), one defines another canonical basis $\{\Xi'_A\mid A\in M(m|n,r)\}$ for $\bsS(m|n,r)$. Note that this basis is not integral basis over $\sZ$ but a cellular basis over $\mbq(\up)$. We now describe its cellularity.

For $\la\in \Lambda(m|n,r),\mu\in \Lambda(m'|n',r)$, let
$$\sD_{\la,\mu}^{+,-}=\sD^+_{(\la^{(0)}|1^{a_1}),(\mu^{(0)}|1^{b_1})}\cap \sD_{(1^{a_0}|\la^{(1)}),(1^{b_0}|\mu^{(1)})},\qquad(\text{see \cite[(3.0.4)]{DR}})$$
where $a_i=|\la^{(i)}|$ and $b_i=|\mu^{(i)}|$. Then the map \eqref{jmath} induces a bijection
$$\jmath^{+,-}:\sD(m|n,r)^{+,-}\longrightarrow M(m|n,r),$$
where
$$\sD(m|n,r)^{+,-}=\{(\la,w,\mu)\mid \la,\mu\in\La(m|n,r),w\in\sD^{+,-}_{\la,\mu}\}.$$
\begin{definition}Let
 $A=\jmath^{+,-}(\al,y,\beta),B=\jmath^{+,-}(\la,w,\mu)\in M(m|n,r)$. Define
$$A\leq_LB \iff y\leq_Lw \mbox{ and } \mu=\beta,$$
where $y\le_Lw$ is the order relation $\le_L$ on $\fS_r$ defined in \cite{KL}.
\end{definition}

With this order relation, the structure constants for the basis $\{\Xi'_A\}_{A\in M(m|n,r)}$ satisfy the following order relation.
\begin{lemma}[{\cite[7.2]{DR}}]\label{DR,7.2}
For $A,B\in M(m|n,r)$, if $\Xi'_A\Xi'_B=\sum_{C\in M(m|n,r)}f_{A,B,C} \Xi'_C$, then $f_{A,B,C}\neq0$ implies $C\leq_LB$ and $C\leq_R A$.
\end{lemma}

Define $A\leq_R B$ if $A^t\leq_L B^t$. Let $\leq_{LR}$ be the preorder generated by $\leq_L$ and $\leq_R$. The relations give rise to three equivalence relations $\sim_L,\sim_R$ and $\sim_{LR}$ on $M(m|n,r)$. Thus, $A\sim_X B$ if and only if $A\leq_X B\leq_X A$ for all $X\in\{L,R,LR\}$. The corresponding equivalence classes in $M(m|n,r)$ with respect to $\sim_L,\sim_R$ and $\sim_{LR}$ are called {\it left cells, right cells} and {\it two-sided cells} respectively.

Like the symmetric group case, the cells defined here can also be described in terms of a super version of the Robinson--Schensted correspondence.

Let  $\Pi(r)$ be the set of all partitions of $r$ and let
$$\Pi(r)_{m|n}=\{\pi\in\Pi(r)\mid \pi_{m+1}\leq n\}.$$
For $\pi\in \Pi(r)_{m|n}$ and
$\mu\in \Lambda(m|n,r)$. A $\pi$-tableau $\sfT$ of content $\mu$ is called a {\it semi-standard $\pi$-supertableau} of type $\mu$ if, in addition,
\begin{itemize}
\item[a)]the entries are weakly increasing in each row and each column of $\sfT$;
\item[b)] the numbers in $\{1,2,\cdots,m\}$ are strictly increasing in the columns and the numbers in $\{m+1,m+2,\cdots,m+n\}$ are strictly increasing in the rows.
\end{itemize}

Let $\mathbf{T}^{su}(\pi,\mu)$ be the set of all semi-stardard $\pi$-supertableaux of content $\mu$. In  particular, for the given partition $\pi$, if we set
\begin{equation}\label{pitilde}
\tilde\pi^{(0)}=(\pi_1,\pi_2,\cdots,\pi_m),\quad \tilde\pi^{(1)}=(\pi_{m+1},\pi_{m+2},\cdots,\pi_{m+n})^t,
\end{equation}
 then $\tilde\pi=(\tilde\pi^{(0)}|\tilde\pi^{(1)})\in\La(m|n,r)$ and $\mathbf{T}^{su}(\pi,\tilde\pi)$ contains a unique element, denoted by $\sfT_\pi$. We also write $\sh(\sfT)=\pi$ if $\sfT\in\mathbf{T}^{su}(\pi,\mu)$, called the {\it shape} of $\sfT$.

\begin{lemma}[{\cite[4.7,7.3]{DR}}]\label{DR,7.3} There is a bijective map
$$\text{\rm RSKs}:M(m|n,r)\longrightarrow\bigcup_{\la,\mu\in\La(m|n,r)\atop \pi\in\Pi(r)_{m|n}}\bfT^{su}(\pi,\la)\times\bfT^{su}(\pi,\mu),\quad A\longrightarrow(\sfp(A),\sfq(A))$$ such that $\la=\ro(A)$, $\mu=co(A)$ and,
for $A,B\in M(m|n,r)$ with $\pi^t=\sh(\sfp(A)), \nu^t=\sh(\sfp(B))$,
\begin{itemize}
\item[(1)] $A\sim_L B$ if  and only if $\sfq(A)=\sfq(B)$.

\item[(2)]$A\sim_R B$ if  and only if $\sfp(A)=\sfp(B)$.

\item[(3)] $A\leq_{LR}B$ implies $\pi\unrhd\nu$. Hence, $A\sim_{LR} B$ if and only if $\sfp(A),\sfp(B)$ have the same shape.
\end{itemize}
\end{lemma}

For $\pi\in \Pi(r)_{m|n}$, let
$$I(\pi)=\bigcup_{\lambda\in\Lambda(m|n,r)}\mathbf{T}^{su}(\pi,\lambda).$$

By the  super RSK correspondence, if  $A\overset{\text{RSKs}}\longrightarrow(\sfS,\sfT)\in I(\pi)$, we relabel the basis element $\Xi'_A$ as $\Xi^{\prime\pi}_{S,T}:=\Xi'_A$.

\begin{lemma}[{\cite[7.4]{DR}}]\label{DR 7.4}
The $\mathbb{Q}(\up)$-basis for $\bsS(m|n,r)$
$$\{\Xi^{\prime\pi}_{\sfS,\sfT}\mid \pi\in \La^+(r)_{m|n},\sfS,\sfT\in I(\pi)\}=\{\Xi'_A\mid A\in M(m|n,r)\}.$$
is a cellular basis in the sense of  \cite{GL}.
\end{lemma}

The cellular basis defines cell modules $C(\pi),\pi\in\Pi(r)_{m|n}$ (see \cite{GL} or \cite[(C.6.3)]{DDPW}). Since
$\bsS(m|n,r)$ is semisimple, all $C(\pi)$ are irreducible.

\begin{theorem} \label{ID}
As a $\bfU(\mathfrak{gl}_{m|n})$-module via $\eta_r:\bfU(\mathfrak{gl}_{m|n})\to \bsS(m|n,r)$, $C(\pi)\cong L(\tilde\pi)$, where $\tilde\pi$ is defined in \eqref{pitilde}.
\end{theorem}
\begin{proof} By the construction, for any fixed $\sfQ\in I(\pi)$, $C(\pi)$ is spanned by $v_\sfS:=\Xi'_{\sfS,\sfQ}+\bsS^{\rhd\pi}$, $\sfS\in I(\pi)$, where $\bsS^{\rhd\pi}$ is spanned by all $\Xi'_A$ with $\sh(\sfp(A))\rhd\pi$.
Let $v_{\tilde\pi}=v_{\sfT_\pi}$. Then the weight of $v_{\tilde\pi}$ is $\tilde \pi$. We now prove that $\tilde\pi$ is the highest weight. It suffices to prove that if $\mathbf{T}^{su}(\pi,\mu)\neq\emptyset$ then $\tilde\pi\unrhd\mu$.

Let $\sfT\in \mathbf{T}^{su}(\pi,\mu)$ and, for $s\in[1,m+n]$, let $\sfT_{\leq s}$ be the subtableau obtained by removing the entries $>s$ and their associated boxes from $\sfT$. If $s\leq m$, then it is known that $\tilde\pi_1+\cdots+\tilde\pi_s\geq\mu_1+\cdots+\mu_s$ (see, e.g., \cite[Lem.~8.42]{DDPW}). Assume now $s>m$. Let $\sfT'_{\leq s}$ be the subtableau consists of top $m$ row of $\sfT_{\leq s}$ and $\sfT_{\leq s}''$ the subtableau obtained by removing $\sfT'_{\leq s}$ from $\sfT_{\leq s}$. We also break $\sfT_\pi$ into two parts $\sfT_{\pi,\leq s}'$ and $\sfT_{\pi,\leq s}''$. Then, by definition, the shape of $\sfT_{\leq s}'$ must be contained in $\sfT_{\pi,\leq s}'$, while the shape of $\sfT_{\leq s}''$ must be contained in $\sfT_{\pi,\leq s}''$. Hence, $\tilde\pi_1+\cdots+\tilde\pi_s\geq\mu_1+\cdots+\mu_s$. This proves the inequality for all $s\ge0$. Hence, $\tilde\pi\unrhd\mu$.
\end{proof}

\begin{remark}\label{remark}
 Unlike the nonsuper case, the cellular basis $\{\Xi'_A\mid A\in M(m|n,r)\}$ does not canonically induce a basis for $C(\pi)$. In other words, the set $\{\Xi'_A\cdot v_{\tilde\pi}\mid A\in M(m|n,r)\}\setminus\{0\}$ does not form a basis for the
cell module $C(\pi)$. This can be seen as follows. Suppose $\Xi'_A=\Xi^{\prime\nu}_{\sfS,\sfT}$. Then
$$0\neq \Xi'_A\cdot v_{\tilde\pi}=\Xi^{\prime\nu}_{\sfS,\sfT}\Xi^{\prime\pi}_{\sfT_\pi,\sfQ}+\bsS^{\rhd\pi}=\sum_{C}f_{C}(A,\sfT_\pi)\Xi^{\prime\pi}_{C}+\bsS^{\rhd\pi}$$ implies $\co(A)=\tilde\pi$, $\tilde\nu\unrhd\tilde\pi$ by the proof above, and $C\leq_RA$ by Lemma \ref{DR,7.2}. Hence,
$\pi\unrhd\nu$ by Lemma \ref{DR,7.3}(3). Thus, we must have $\tilde\pi^{(0)}=\tilde\nu^{(0)}$ and $\tilde\pi^{(1)}\unlhd\tilde\nu^{(1)}$ (equivalently, $\pi^{(1)}\unrhd\nu^{(1)}$). Cosequently, we do not have $\nu=\pi$ in general unless $n=1$ and so the cardinality of the set could be larger than $\dim C(\pi)$.
\end{remark}

\begin{corollary}\label{primitive vector} We have
$\sS(m|n,r)^+v_{\tilde{\pi}}=0$. In other words, by regarding $C(\pi)$ as a $\bfU(\mathfrak{gl}_{m|n})$-module, $v_{\tilde\pi}$ is a primitive vector.
\end{corollary}
\begin{proof} We first observe that, if $\co(E_{h,h+1}+\diag(\la))=\tilde\pi$, i.e., $\la+\bse_{h+1}=\tilde\pi$, then $\la=\tilde\pi-\bse_{h+1}$. Since  $E_{h,h+1}(\bfl,r)\cdot v_{\tilde\pi}=[E_{h,h+1}+\diag(\la)]v_{\tilde\pi}$ has weight $\tilde\pi+\bse_h-\bse_{h+1}$ and $\tilde\pi+\bse_h-\bse_{h+1}\rhd\tilde\pi$, we must have $E_{h,h+1}(\bfl,r)\cdot v_{\tilde\pi}=0$ by the theorem above.
\end{proof}

\begin{corollary}\label{canonical basis of simple module1} Let $M(m|n,r)^{\leq0}=\{(a_{i,j})\in M(m|n,r)\mid a_{i,j}=0\;\forall i<j\}$. Then
the set $\{\Xi_A\cdot v_{\tilde\pi}\mid A\in M(m|n,r)^{\leq0}\}$ spans the cell module $C(\pi)$.
\end{corollary}

It would be interesting to extract a basis for $C(\pi)$ from this spanning set in some ``canonical'' way. This is because such a basis can also be induced from the canonical basis $\mathscr C^-$ of $U^-_\sZ(\mathfrak{gl}_{m|n})$ as shown in the following result.

For $\pi\in\Pi(r)_{m|n}$, define a subset $M(m|n,\pi)$ of $M(m|n,r)^{\leq0}$ by the condition that
 $$\{\Xi_A\cdot v_{\tilde\pi}\mid A\in M(m|n,\pi)\}\text{ forms a basis for $C(\pi)$.}$$
  Note that $\co(A)=\tilde\pi$ for $A\in M(m|n,\pi)$. Recall  the notation $A_\la$ defined in \eqref{Ala}.

\begin{theorem}\label{ngeq1}
Let $\mathscr C^-$ be the canonical basis for $U^-_\sZ(\mathfrak{gl}_{m|n})$ and let $L(\mu)$ be a simple polynomial representations of $\bfU(\mathfrak{gl}_{m|n})$ at level $r\ge0$. Then there exists a partition $\pi\in\Pi(r)_{m|n}$ such that $\mu=\tilde\pi$ and
$$\{\sfC_A\cdot v_{\tilde\pi}\mid A\in M(m|n)^-, A_{\tilde\pi}\in M(m|n,\pi)\}$$
forms a basis for $L(\mu)$.
\end{theorem}

\begin{proof} The first assertion follows from Lemma \ref{inflate} and Theorem \ref{ID}. Thus, $L(\mu)\cong C(\pi)$. By Theorem \ref{canonical basis of two parts},
$$\sfC_A\cdot v_{\tilde\pi}=\sfc_A\cdot v_{\tilde\pi}=\sum_{\la:\la\geq\bsfh(A)}(-1)^{\bar A_\la}\Xi_{A_\la}\cdot v_{\tilde\pi}=(-1)^{\bar A_{\tilde\pi}}\Xi_{A_{\tilde\pi}}\cdot v_{\tilde\pi}.$$
The last assertion follows from the definition of $M(m|n,\pi)$.
\end{proof}

We end the paper with a canonical description of $M(m|n,\pi)$ for $n=1$. This case has already been considered in \cite{CHW2} is a natural application of the observation in Remark \ref{remark}.

%In this case, the canonical basis $\mathscr C^-$ for $U^-_\sZ(\mathfrak{gl}_{m|n})$ does induces bases for simple polynomial representations in a way similar to the nonsuper case.

\begin{theorem}\label{n=1}
For $\pi\in  \Pi(r)_{m|1}$, the set $\{\Xi_A\cdot v_\pi\mid A\in M(m|1,r)\}\setminus\{0\}$ forms a basis for the
cell module $C(\pi)$.
\end{theorem}
\begin{proof} We first claim that the set $\{\Xi'_A\cdot v_\pi\mid A\in M(m|1,r)\}\setminus\{0\}$ forms a basis for $C(\pi)$. Indeed,
assume $\Xi'_A=\Xi^{\prime\nu}_{\sfS,\sfT}$. By  Remark \ref{remark},
$\Xi'_A\cdot v_{\tilde\pi}\neq0$
implies $\co(A)=\tilde\pi$, $\pi\unrhd\nu$ and  $\tilde\pi^{(0)}=\tilde\nu^{(0)}$. Since $|\pi|=|\nu|$ and $n=1$, we must have $\pi_{m+1}=\nu_{m+1}$, forcing $\pi=\nu$ and $\sfT=\sfT_\pi$. Thus, a dimensional comparison proves the claim.

Further, for $\pi\in \Pi(r)_{m|1}$, $\pi_{m+1}\leq 1$. This forces the subgroup $\fS_\beta$ defined in \eqref{beta} is trivial. Hence,  $\xi'_A=\xi_A$. By the argument given around  \cite[Rem.~6.5]{DR},  we also have $\Xi'_A=\Xi_A$ whenever $\co(A)=\tilde\pi$. Now the result follows from the claim above.
\end{proof}

With this theorem, the index set $M(m|1,\pi)$ can have the following canonical description:
 $$\aligned
M(m|1,\pi)&=\{A\in M(m|1,r)\mid \Xi'_A \cdot v_\pi\neq 0\}\\
&=\{A\in M(m|1,r)\mid \sfp(A)\in I(\pi), \sfq(A)=\sfT_\pi\}.
\endaligned  $$
Theorems  \ref{ngeq1} and \ref{n=1} gives immediately the following.
\begin{corollary}\label{CHWConj}
Let $\mathscr C^-=\{\cC_A\mid A\in M(m|1)^-\}$ be the canonical basis for $U^-_\sZ(\mathfrak{gl}_{m|1})$ as given in \eqref{C-} and let $L(\mu)$ be a simple polynomial representations of $\bfU(\mathfrak{gl}_{m|1})$ with highest weight vector $v_\mu$.  Then
$$\{\sfC_A\cdot v_\mu\mid A\in M(m|1)^-\}\setminus \{0\}$$
forms a  basis for $L(\mu)$.
\end{corollary}

We have proved the conjecture \cite[Conj.~8.9]{CHW2} for polynomial representations.

\vspace{.3cm}
\noindent
{\bf Acknowledgement.} The authors would like to thank Weiqiang Wang for the reference \cite{CHW2}. The first author also thanks him for various discussions during his visit to Charlottesville in January 2014 and for his comments on the canonical property in the $\mathfrak{gl}_{m|1}$ case.

\end{document}